\documentclass[11pt]{amsart} 
\usepackage{amssymb}
\newtheorem{theorem}{Theorem}
\newtheorem{lemma}{Lemma}
\newtheorem{corollary}{Corollary}
\newtheorem{proposition}{Proposition}

\def\End{\mathop{\mathrm{End}}\nolimits}
\def\Hom{\mathop{\mathrm{Hom}}\nolimits}

\begin{document}
\title[Zero-Energy Fields on Complex Projective Space]
{Zero-Energy Fields on\\ Complex Projective Space}
\author{Michael Eastwood}
\address{Mathematical Sciences Institute, Australian National 
University,\newline ACT 0200, Australia}
\email{meastwoo@member.ams.org}
\author{Hubert Goldschmidt}
\address{Department of Mathematics, Columbia University,\newline 
New York, NY 10027, USA}
\email{hg@math.columbia.edu}
\subjclass{Primary 53C65; Secondary 53B35, 53D05, 58J10, 58J70.}
\keywords{Integral geometry, X-ray transform.}
\begin{abstract}
We consider complex projective space with its Fubini--Study metric and the 
X-ray transform defined by integration over its geodesics. We identify the 
kernel of this transform acting on symmetric tensor fields.
\end{abstract}
\renewcommand{\subjclassname}{\textup{2000} Mathematics Subject Classification}
\maketitle
\renewcommand{\thefootnote}{}
\footnotetext{Some of this work was done during the 2006 Summer Program at the
Institute for Mathematics and its Applications at the University of Minnesota.
The authors would like to thank the IMA for hospitality during this time. 
Eastwood is supported by the Australian Research Council.}

\section{Introduction}
Suppose $\omega_{ab\cdots c}$ is a smooth symmetric covariant tensor field
defined on a Riemannian manifold~$M$. Suppose $\gamma$ is a smooth oriented
curve on $M$ joining points $p$ and~$q$. Let $X^a$ denote the unit vector field
defined along $\gamma$ and tangent to $\gamma$ consistent with its orientation.
We obtain a real number $\int_\gamma\omega_{ab\cdots c}$ by integrating the
function $X^aX^b\cdots X^c\omega_{ab\cdots c}$ along $\gamma$ with respect to
arc-length. 

Suppose that $\phi_{b\cdots c}$ is a symmetric covariant tensor field such 
that $\omega_{ab\cdots c}=\nabla_{(a}\phi_{b\cdots c)}$ where $\nabla_a$ is 
the Levi-Civita connection and round brackets denote symmetrisation over the 
indices they enclose. Suppose $\gamma$ is a geodesic. This means that 
$X^a\nabla_a X^b=0$ and, in this case, 
\[X^a\nabla_a(X^b\cdots X^c\phi_{b\cdots c})=
X^aX^b\cdots X^c\nabla_{(a}\phi_{b\cdots c)}=
X^aX^b\cdots X^c\omega_{ab\cdots c}.\]
Therefore 
$\int_\gamma\omega_{ab\cdots c}= \big[X^b\cdots X^c\phi_{b\cdots c}\big]_p^q$
and, in particular, if $\gamma$ is a closed geodesic, then
$\int_\gamma\omega_{ab\cdots c}=0$. On complex projective space
${\mathbb{CP}}_n$ with its standard Fubini--Study metric~\cite{besse,cannas},
all geodesics are closed and the {\em X-ray transform\/} on symmetric tensor 
fields associates to $\omega_{ab\cdots c}$ the function
\[\gamma\longmapsto\int_\gamma\omega_{ab\cdots c}\]
defined on the space of geodesics on~${\mathbb{CP}}_n$. We shall refer to 
fields in the kernel of this transform as having {\em zero energy}. We have 
just observed that fields of the form $\nabla_{(a}\phi_{b\cdots c)}$ have zero 
energy. The main aim of this article is to prove the converse, namely
\begin{theorem}\label{maintheorem} On ${\mathbb{CP}}_n$ for $n\geq 2$, a smooth
symmetric covariant tensor field $\omega_{ab\cdots c}$ of valence $p\geq 1$
having zero energy must be of the form $\nabla_{(a}\phi_{b\cdots c)}$ for some
smooth symmetric field $\phi_{b\cdots c}$ of valence $p-1$.
\end{theorem}

This theorem was first proved for $p=1$ in~\cite{gg5}. In case $p=2$, it was
first proved by Tsukamoto~\cite{t}; other proofs in this case can be found
in~\cite{gg2,gg3} and Chapter~III of~\cite{gg6}. Tsukamoto's proof for $p=2$
heavily relied on harmonic analysis on~${\mathbb{CP}}_2$.
In~\cite[Theorem~3.40]{gg6} harmonic analysis on complex projective space was
eliminated from the proof of case $p=1$ (and already in~\cite{gg3} harmonic
analysis was severely reduced for the case $p=2$). In a proof of
Theorem~\ref{maintheorem} for $n=2$ given in~\cite{rims}, harmonic analysis on
${\mathbb{CP}}_2$ arose in the guise of twistor theory; this proof relied on
the Fubini-Study metric on ${\mathbb{CP}}_2$ being (anti-)self-dual and was
therefore limited to the case $n=2$. In this article, our proof is uniform
for all $n$ and~$p$; it completely eliminates any harmonic analysis
on~${\mathbb{CP}}_n$.

Inspired by a remark of J.-P. Demailly (cf.~\cite[Introduction]{gg3}) in 
case $p=1$,
our plan is to deduce Theorem~\ref{maintheorem} from the corresponding statement for real projective
space ${\mathbb{RP}}_n$ with its usual round metric (inherited from the round
$n$-sphere). The truth of this statement for ${\mathbb{RP}}_n$ has been shown
by various means~\cite{be,srni,e,gg6,g1,m1,m2}. (The precise method of proof
for ${\mathbb{RP}}_n$ will not enter our discussion for~${\mathbb{CP}}_n$.) The
point is that the standard embedding
${\mathbb{RP}}_n\hookrightarrow{\mathbb{CP}}_n$ induced by
${\mathbb{R}}^{n+1}\hookrightarrow{\mathbb{C}}^{n+1}$ is totally geodesic.
Furthermore, all translates of this standard
${\mathbb{RP}}_n\hookrightarrow{\mathbb{CP}}_n$ by ${\mathrm{SU}}(n+1)$, the
isometry group of~${\mathbb{CP}}_n$, are totally geodesic and so this provides
many submanifolds on which the kernel of the X-ray transform is known. It is
immediate, for example, that injectivity of the X-ray transform on smooth
functions on ${\mathbb{RP}}_n$ implies that the same is true
on~${\mathbb{CP}}_n$. More generally, we shall require some algebraic link
between tensors on ${\mathbb{CP}}_n$ and on these embedded real projective
spaces. Such a link is the subject of the following two sections.

\section{Some linear algebra} Let us call a tensor bundle on a smooth
$2n$-dimensional manifold $M$ {\em irreducible\/} if and only if it is induced
from the co-frame bundle by an irreducible representation of
${\mathrm{SL}}(2n,{\mathbb{R}})$. Sections of such a bundle will be called 
{\em irreducible tensors\/}. Now the irreducible representations of
${\mathrm{SL}}(2n,{\mathbb{R}})$ are classified by their highest
weight~\cite{fh}, which we may write as an integral combination of fundamental
weights, the coefficients of which may be written over the corresponding nodes
of the Dynkin diagram. Let us restrict attention to those tensor bundles
arising from representations of the form
\begin{equation}\label{nottootall}\begin{picture}(160,10)
\put(0,0){\line(1,0){30}}
\put(40,0){\makebox(0,0){$\cdots$}}
\put(50,0){\line(1,0){60}}
\put(120,0){\makebox(0,0){$\cdots$}}
\put(130,0){\line(1,0){30}}
\put(0,0){\makebox(0,0){$\bullet$}}
\put(0,7){\makebox(0,0){\scriptsize$a_1$}}
\put(20,0){\makebox(0,0){$\bullet$}}
\put(20,7){\makebox(0,0){\scriptsize$a_2$}}
\put(60,0){\makebox(0,0){$\bullet$}}
\put(60,7){\makebox(0,0){\scriptsize$a_{n-1}$}}
\put(80,0){\makebox(0,0){$\bullet$}}
\put(80,7){\makebox(0,0){\scriptsize$a_n$}}
\put(100,0){\makebox(0,0){$\bullet$}}
\put(100,7){\makebox(0,0){\scriptsize$0$}}
\put(140,0){\makebox(0,0){$\bullet$}}
\put(140,7){\makebox(0,0){\scriptsize$0$}}
\put(160,0){\makebox(0,0){$\bullet$}}
\put(160,7){\makebox(0,0){\scriptsize$0$}}
\end{picture}\;.\end{equation}
These representations have the property that their restriction to the subgroup  
${\mathrm{Sp}}(2n,{\mathbb{R}})\subset{\mathrm{SL}}(2n,{\mathbb{R}})$ has a 
leading term
\begin{equation}\label{leadingterm}\begin{picture}(160,10)
\put(0,0){\line(1,0){30}}
\put(40,0){\makebox(0,0){$\cdots$}}
\put(50,0){\line(1,0){60}}
\put(120,0){\makebox(0,0){$\cdots$}}
\put(130,0){\line(1,0){30}}
\put(0,0){\makebox(0,0){$\bullet$}}
\put(0,7){\makebox(0,0){\scriptsize$a_1$}}
\put(20,0){\makebox(0,0){$\bullet$}}
\put(20,7){\makebox(0,0){\scriptsize$a_2$}}
\put(60,0){\makebox(0,0){$\bullet$}}
\put(60,7){\makebox(0,0){\scriptsize$a_{n-1}$}}
\put(80,0){\makebox(0,0){$\bullet$}}
\put(80,7){\makebox(0,0){\scriptsize$a_n$}}
\put(100,0){\makebox(0,0){$\bullet$}}
\put(100,7){\makebox(0,0){\scriptsize$0$}}
\put(140,0){\makebox(0,0){$\bullet$}}
\put(140,7){\makebox(0,0){\scriptsize$0$}}
\put(160,0){\makebox(0,0){$\bullet$}}
\put(160,7){\makebox(0,0){\scriptsize$0$}}
\end{picture}\enskip\raisebox{-2pt}{$=$}\enskip
\begin{picture}(85,10)
\put(0,0){\line(1,0){30}}
\put(40,0){\makebox(0,0){$\cdots$}}
\put(50,0){\line(1,0){10}}
\put(60,1){\line(1,0){26}}
\put(60,-1){\line(1,0){26}}
\put(0,0){\makebox(0,0){$\bullet$}}
\put(0,7){\makebox(0,0){\scriptsize$a_1$}}
\put(20,0){\makebox(0,0){$\bullet$}}
\put(20,7){\makebox(0,0){\scriptsize$a_2$}}
\put(60,0){\makebox(0,0){$\bullet$}}
\put(60,7){\makebox(0,0){\scriptsize$a_{n-1}$}}
\put(86,0){\makebox(0,0){$\bullet$}}
\put(86,7){\makebox(0,0){\scriptsize$a_n$}}
\put(73,0){\makebox(0,0){$\langle$}}
\end{picture}\enskip\;
\raisebox{-2pt}{$\oplus\;\cdots\;,$}\end{equation}
which is easily described in terms of tensors, namely
\begin{equation}\label{totallytracefreepart}
\psi_{abc\cdots d}=\psi_{abc\cdots d}^\perp+\mbox{terms of the form }
J_{ab}\bowtie\theta_{c\cdots d}.\end{equation}
Here, $J_{ab}$ is the non-degenerate skew form preserved by
${\mathrm{Sp}}(2n,{\mathbb{R}})$, the tensor $\psi_{abc\cdots d}^\perp$
satisfies the same symmetries as does $\psi_{abc\cdots d}$ but is, in addition,
totally trace-free with respect to the inverse form $J^{ab}$, and 
$\bowtie$ is some
symmetry operation on the indices $abc\cdots d$. Suppose, for example, that
$R_{abcd}$ has Riemann tensor symmetries. The relevant branching
for ${\mathrm{Sp}}(2n,{\mathbb{R}})\subset{\mathrm{SL}}(2n,{\mathbb{R}})$ is
\[\begin{picture}(75,10)
\put(0,0){\line(1,0){35}}
\put(45,0){\makebox(0,0){$\cdots$}}
\put(55,0){\line(1,0){20}}
\put(0,0){\makebox(0,0){$\bullet$}}
\put(0,7){\makebox(0,0){\scriptsize$0$}}
\put(15,0){\makebox(0,0){$\bullet$}}
\put(15,7){\makebox(0,0){\scriptsize$2$}}
\put(30,0){\makebox(0,0){$\bullet$}}
\put(30,7){\makebox(0,0){\scriptsize$0$}}
\put(60,0){\makebox(0,0){$\bullet$}}
\put(60,7){\makebox(0,0){\scriptsize$0$}}
\put(75,0){\makebox(0,0){$\bullet$}}
\put(75,7){\makebox(0,0){\scriptsize$0$}}
\end{picture}\enskip\raisebox{-2pt}{$=$}\enskip
\begin{picture}(60,10)
\put(0,0){\line(1,0){20}}
\put(30,0){\makebox(0,0){$\cdots$}}
\put(40,0){\line(1,0){5}}
\put(45,1){\line(1,0){15}}
\put(45,-1){\line(1,0){15}}
\put(0,0){\makebox(0,0){$\bullet$}}
\put(0,7){\makebox(0,0){\scriptsize$0$}}
\put(15,0){\makebox(0,0){$\bullet$}}
\put(15,7){\makebox(0,0){\scriptsize$2$}}
\put(45,0){\makebox(0,0){$\bullet$}}
\put(45,7){\makebox(0,0){\scriptsize$0$}}
\put(60,0){\makebox(0,0){$\bullet$}}
\put(60,7){\makebox(0,0){\scriptsize$0$}}
\put(52,0){\makebox(0,0){$\langle$}}
\end{picture}\enskip\;
\raisebox{-2pt}{$\oplus$}\;\enskip
\begin{picture}(60,10)
\put(0,0){\line(1,0){20}}
\put(30,0){\makebox(0,0){$\cdots$}}
\put(40,0){\line(1,0){5}}
\put(45,1){\line(1,0){15}}
\put(45,-1){\line(1,0){15}}
\put(0,0){\makebox(0,0){$\bullet$}}
\put(0,7){\makebox(0,0){\scriptsize$0$}}
\put(15,0){\makebox(0,0){$\bullet$}}
\put(15,7){\makebox(0,0){\scriptsize$1$}}
\put(45,0){\makebox(0,0){$\bullet$}}
\put(45,7){\makebox(0,0){\scriptsize$0$}}
\put(60,0){\makebox(0,0){$\bullet$}}
\put(60,7){\makebox(0,0){\scriptsize$0$}}
\put(52,0){\makebox(0,0){$\langle$}}
\end{picture}\enskip\;
\raisebox{-2pt}{$\oplus$}\;\enskip
\begin{picture}(60,10)
\put(0,0){\line(1,0){20}}
\put(30,0){\makebox(0,0){$\cdots$}}
\put(40,0){\line(1,0){5}}
\put(45,1){\line(1,0){15}}
\put(45,-1){\line(1,0){15}}
\put(0,0){\makebox(0,0){$\bullet$}}
\put(0,7){\makebox(0,0){\scriptsize$0$}}
\put(15,0){\makebox(0,0){$\bullet$}}
\put(15,7){\makebox(0,0){\scriptsize$0$}}
\put(45,0){\makebox(0,0){$\bullet$}}
\put(45,7){\makebox(0,0){\scriptsize$0$}}
\put(60,0){\makebox(0,0){$\bullet$}}
\put(60,7){\makebox(0,0){\scriptsize$0$}}
\put(52,0){\makebox(0,0){$\langle$}}
\end{picture}
\]
and may be written explicitly as 
\begin{equation}\label{symplecticRiemann}
\raisebox{3pt}{\makebox[0pt]{$\begin{array}{rcl}R_{abcd}&\!\!=\!\!&X_{abcd}\\
&&\hspace*{-8pt}{}+\Psi_{ac}J_{bd}-\Psi_{bc}J_{ad}
-\Psi_{ad}J_{bc}+\Psi_{bd}J_{ac}
+2\Psi_{ab}J_{cd}+2\Psi_{cd}J_{ab}\\
&&\hspace*{-8pt}{}+L(J_{ac}J_{bd}-J_{bc}J_{ad}+2J_{ab}J_{cd}),
\end{array}$}}\end{equation}
where $X_{abcd}$ has Riemann tensor symmetries and is trace-free whilst 
$\Psi_{ab}$ is skew and trace-free (where `trace-free' means with respect 
to $J^{ab}$). It is the symplectic counterpart to the well-known decomposition 
\[R_{abcd}=W_{abcd}+\Phi_{ac}g_{bd}-\Phi_{bc}g_{ad}
-\Phi_{ad}g_{bc}+\Phi_{bd}g_{ac}
+K(g_{ac}g_{bd}-g_{bc}g_{ad})\]
of the Riemann tensor under 
${\mathrm{SO}}(2n,{\mathbb{R}})\subset{\mathrm{SL}}(2n,{\mathbb{R}})$.
\begin{proposition}\label{One}
Suppose $\psi_{abc\cdots d}$ is an irreducible covariant tensor under
${\mathrm{SL}}(2n,{\mathbb{R}})$ with symmetries of the
form~{\rm(\ref{nottootall})}. Then its totally trace-free part 
$\psi_{abc\cdots d}^\perp$, defined by {\rm(\ref{totallytracefreepart})}, 
vanishes if and only if the pullback of $\psi_{abc\cdots d}$ to every 
Lagrangian subspace of ${\mathbb{R}}^{2n}$ vanishes.
\end{proposition}
\begin{proof}
Considering the right hand side of~(\ref{totallytracefreepart}), it is clear
that that all terms except $\psi_{abc\cdots d}^\perp$ vanish when restricted to
a Lagrangian subspace simply because $J_{ab}$ has this property. Conversely,
requiring that $\psi_{abc\cdots d}$ vanish on all Lagrangian subspaces is a
manifestly ${\mathrm{Sp}}(2n,{\mathbb{R}})$-invariant restriction. Bearing in
mind that the leading term of (\ref{leadingterm}) is irreducible, it follows
that either our proposition is true or all tensors with symmetries of the form
(\ref{nottootall}) vanish on all Lagrangian subspaces. If we firstly consider
fundamental representations of the form (\ref{nottootall}), then we are done 
because the corresponding tensors are precisely the $k$-forms for $k\leq n$. 
More specifically, we can choose a basis 
$\{e_1,e_2,\cdots,e_n,e_{n+1},e_{n+2},\cdots,e_{2n}\}$
of ${\mathbb{R}}^{2n}$ such that 
\[J_{ab}=\left[\begin{array}{c|c}0&{\mathrm{Id}}\\ \hline
\rule{0pt}{10pt}-{\mathrm{Id}}&0
\end{array}\right]\]
and consider the Lagrangian subspace
\[\Pi\equiv{\mathrm{span}}\{e_1,e_2,\cdots,e_n\},\]
noticing that the highest weight vector $\omega^k\in\Lambda^k{\mathbb{R}}^{2n}$
restricts to a non-zero form on~$\Pi$. The general case follows because  
\[(\omega^1)^{\otimes a_i}\otimes(\omega^2)^{\otimes a_2}\otimes\cdots
\otimes(\omega^n)^{\otimes a_n}\in
\enskip
{\setlength{\unitlength}{.9pt}
\begin{picture}(160,10)
\put(0,0){\line(1,0){30}}
\put(40,0){\makebox(0,0){$\cdots$}}
\put(50,0){\line(1,0){60}}
\put(120,0){\makebox(0,0){$\cdots$}}
\put(130,0){\line(1,0){30}}
\put(0,0){\makebox(0,0){$\bullet$}}
\put(0,7){\makebox(0,0){\scriptsize$a_1$}}
\put(20,0){\makebox(0,0){$\bullet$}}
\put(20,7){\makebox(0,0){\scriptsize$a_2$}}
\put(60,0){\makebox(0,0){$\bullet$}}
\put(60,7){\makebox(0,0){\scriptsize$a_{n-1}$}}
\put(80,0){\makebox(0,0){$\bullet$}}
\put(80,7){\makebox(0,0){\scriptsize$a_n$}}
\put(100,0){\makebox(0,0){$\bullet$}}
\put(100,7){\makebox(0,0){\scriptsize$0$}}
\put(140,0){\makebox(0,0){$\bullet$}}
\put(140,7){\makebox(0,0){\scriptsize$0$}}
\put(160,0){\makebox(0,0){$\bullet$}}
\put(160,7){\makebox(0,0){\scriptsize$0$}}
\end{picture}}\]
is non-zero when restricted to~$\Pi$. 
\end{proof}

\section{Symplectic geometry on complex projective space}
In this article, complex projective space may always be regarded as a Riemannian 
manifold with its Fubini-Study metric, which we shall denote by~$g_{ab}$. 
{From} this point of view
\begin{equation}\label{homogeneous}
{\mathbb{CP}}_n={\mathrm{SU}}(n+1)/
{\mathrm{S}}({\mathrm{U}}(1)\times {\mathrm{U}}(n)).\end{equation}
However, ${\mathbb{CP}}_n$ may also be viewed in other well-known guises as
follows.
\[\begin{tabular}{|c|c|c|c|}\hline
Structure&Quantity&Name&Formula\\ \hline\hline
Riemannian &$g_{ab}$&Fubini-Study metric &
\rule{0pt}{12pt} $g_{ab}=J_a{}^cJ_{bc}$\\ \hline
complex &$J_a{}^b$ & complex structure &  
\rule{0pt}{12pt} $J_a{}^b=g^{bc}J_{ac}$\\ \hline
symplectic &$J_{ab}$ & K\"ahler form &  
\rule{0pt}{12pt} $J_{ab}=J_a{}^cg_{bc}$\\ \hline
\end{tabular}\]
The formul\ae\ show that any two of these structures determine the third. As
already remarked in the Introduction, there is a useful family of totally
geodesic embeddings ${\mathbb{RP}}_n\hookrightarrow{\mathbb{CP}}_n$ obtained
from the standard embedding by the action of the isometry
group~${\mathrm{SU}}(n+1)$. For want of a better terminology, let us refer to
these as {\em model embeddings}.
\begin{proposition}\label{Two}
Suppose that ${\mathbb{RP}}_n\hookrightarrow{\mathbb{CP}}_n$ is a model
embedding. Then $T_p{\mathbb{RP}}_n\hookrightarrow T_p{\mathbb{CP}}_n$ is a
Lagrangian subspace for all $p\in{\mathbb{RP}}_n$. Conversely, for each
$p\in{\mathbb{CP}}_n$ the model embeddings passing through $p$ are in 1--1
correspondence with the Lagrangian subspaces of~$T_p{\mathbb{CP}}_n$.
\end{proposition}
\begin{proof}
The K\"ahler form $J_{ab}$ is of type $(1,1)$ and therefore vanishes on any
totally real submanifold of~${\mathbb{CP}}_n$. In particular, it vanishes on
any model embedding. Thus, these embeddings are 
Lagrangian. Conversely, we may as well take $p$ to be the basepoint of 
${\mathbb{CP}}_n$ as the homogeneous space (\ref{homogeneous}), in which case
it is easily calculated that the isotropy group acts on $T_p{\mathbb{CP}}_n$ 
as 
\[{\mathrm{S}}({\mathrm{U}}(1)\times{\mathrm{U}}(n))\ni(\lambda,A)\mapsto
\lambda^{-1}A\mbox{ acting on }{\mathbb{C}}^n.\]
In particular, every transformation in ${\mathrm{U}}(n)$ is obtained in this 
way. Also note that the K\"ahler form is realised as the standard 
symplectic form on~${\mathbb{C}}^n$. Therefore, it remains to be seen that 
${\mathrm{U}}(n)$ acts transitively on the Lagrangian subspaces 
of~${\mathbb{C}}^n$. This is a well-known fact: see, e.g. 
\cite[Exercise~I.A.4(ii)]{h}. (Otherwise said, the Lagrangian Grassmannian may 
be realised as the homogeneous space~${\mathrm{U}}(n)/{\mathrm{O}}(n)$.)
\end{proof}

\begin{theorem}\label{minortheorem}
Let $\psi_{abc\cdots d}$ be an irreducible tensor on
${\mathbb{CP}}_n$ corresponding to a representation of\/
${\mathrm{SL}}(2n,{\mathbb{R}})$ of the form~{\rm(\ref{nottootall})}. Suppose
that $\psi_{abc\cdots d}$ vanishes when restricted to all model embeddings
${\mathbb{RP}}_n\hookrightarrow{\mathbb{CP}}_n$. Then its totally trace-free
part $\psi_{abc\cdots d}^\perp$ defined by {\rm(\ref{totallytracefreepart})}
vanishes.\end{theorem}
\begin{proof} Immediate by combining Propositions~\ref{One} and~\ref{Two}. 
\end{proof}
\begin{corollary} A smooth two-form on ${\mathbb{CP}}_n$ vanishes upon
restriction to every model embedding
${\mathbb{RP}}_n\hookrightarrow{\mathbb{CP}}_n$ if and only if it is of the
form $\theta J_{ab}$, where $\theta$ is a some smooth function and $J_{ab}$ is 
the K\"ahler form.
\end{corollary}
\begin{proof} If $n=1$ then the hypothesis and conclusion are always trivially
satisfied. For $n\geq 2$ the representation of
${\mathrm{SL}}(2n,{\mathbb{R}})$ corresponding to two-forms is 
\[\begin{picture}(75,10)
\put(0,0){\line(1,0){35}}
\put(45,0){\makebox(0,0){$\cdots$}}
\put(55,0){\line(1,0){20}}
\put(0,0){\makebox(0,0){$\bullet$}}
\put(0,7){\makebox(0,0){\scriptsize$0$}}
\put(15,0){\makebox(0,0){$\bullet$}}
\put(15,7){\makebox(0,0){\scriptsize$1$}}
\put(30,0){\makebox(0,0){$\bullet$}}
\put(30,7){\makebox(0,0){\scriptsize$0$}}
\put(60,0){\makebox(0,0){$\bullet$}}
\put(60,7){\makebox(0,0){\scriptsize$0$}}
\put(75,0){\makebox(0,0){$\bullet$}}
\put(75,7){\makebox(0,0){\scriptsize$0$}}
\end{picture}\quad(\geq 3\mbox{ nodes})\]
and Theorem~\ref{minortheorem} applies.\end{proof}
For the purposes of this article, we shall need Theorem~\ref{minortheorem} for
tensors having symmetries of the form 
\begin{equation}\label{generalisedRiemann}\begin{picture}(75,10)
\put(0,0){\line(1,0){35}}
\put(45,0){\makebox(0,0){$\cdots$}}
\put(55,0){\line(1,0){20}}
\put(0,0){\makebox(0,0){$\bullet$}}
\put(0,7){\makebox(0,0){\scriptsize$0$}}
\put(15,0){\makebox(0,0){$\bullet$}}
\put(15,7){\makebox(0,0){\scriptsize$\ell$}}
\put(30,0){\makebox(0,0){$\bullet$}}
\put(30,7){\makebox(0,0){\scriptsize$0$}}
\put(60,0){\makebox(0,0){$\bullet$}}
\put(60,7){\makebox(0,0){\scriptsize$0$}}
\put(75,0){\makebox(0,0){$\bullet$}}
\put(75,7){\makebox(0,0){\scriptsize$0$}}
\end{picture}\quad(\geq 3\mbox{ nodes})\end{equation}
for $\ell\geq 1$ and it is worthwhile stating what this means more explicitly
in this case. Writing square brackets to denote skew-symmetrisation over the 
indices they enclose, the tensors themselves may realised in the form
\begin{equation}\label{skewpairs}R_{paqb\cdots rc}=R_{[pa][qb]\cdots [rc]},
\end{equation}
with $\ell$ pairs of skew indices, symmetric in these pairs, and such that
\begin{equation}\label{bianchi}R_{[paq]b\cdots rc}=0,\end{equation}
generalising the symmetries of a Riemann tensor, which is the case~$\ell=2$.
Such tensors enjoy a decomposition 
\[R_{paqb\cdots rc}=X_{paqb\cdots rc}+J\mbox{-trace terms}\] 
generalising (\ref{symplecticRiemann}), where $X_{paqb\cdots rc}$ is totally 
trace-free and the $J$-trace terms follow the branching 
\[\begin{picture}(90,10)
\put(0,0){\line(1,0){35}}
\put(45,0){\makebox(0,0){$\cdots$}}
\put(55,0){\line(1,0){35}}
\put(0,0){\makebox(0,0){$\bullet$}}
\put(0,7){\makebox(0,0){\scriptsize$0$}}
\put(15,0){\makebox(0,0){$\bullet$}}
\put(15,7){\makebox(0,0){\scriptsize$\ell$}}
\put(30,0){\makebox(0,0){$\bullet$}}
\put(30,7){\makebox(0,0){\scriptsize$0$}}
\put(60,0){\makebox(0,0){$\bullet$}}
\put(60,7){\makebox(0,0){\scriptsize$0$}}
\put(75,0){\makebox(0,0){$\bullet$}}
\put(75,7){\makebox(0,0){\scriptsize$0$}}
\put(90,0){\makebox(0,0){$\bullet$}}
\put(90,7){\makebox(0,0){\scriptsize$0$}}
\end{picture}\enskip\raisebox{-2pt}{$=$}\enskip\begin{picture}(75,10)
\put(0,0){\line(1,0){35}}
\put(45,0){\makebox(0,0){$\cdots$}}
\put(55,0){\line(1,0){5}}
\put(60,1){\line(1,0){15}}
\put(60,-1){\line(1,0){15}}
\put(0,0){\makebox(0,0){$\bullet$}}
\put(0,7){\makebox(0,0){\scriptsize$0$}}
\put(15,0){\makebox(0,0){$\bullet$}}
\put(15,7){\makebox(0,0){\scriptsize$\ell$}}
\put(30,0){\makebox(0,0){$\bullet$}}
\put(30,7){\makebox(0,0){\scriptsize$0$}}
\put(60,0){\makebox(0,0){$\bullet$}}
\put(60,7){\makebox(0,0){\scriptsize$0$}}
\put(75,0){\makebox(0,0){$\bullet$}}
\put(75,7){\makebox(0,0){\scriptsize$0$}}
\put(67,0){\makebox(0,0){$\langle$}}
\end{picture}\enskip
\raisebox{-2pt}{$\oplus\bigoplus_{j=\ell-1}^{j=0}$}\;
\begin{picture}(75,10)
\put(0,0){\line(1,0){35}}
\put(45,0){\makebox(0,0){$\cdots$}}
\put(55,0){\line(1,0){5}}
\put(60,1){\line(1,0){15}}
\put(60,-1){\line(1,0){15}}
\put(0,0){\makebox(0,0){$\bullet$}}
\put(0,7){\makebox(0,0){\scriptsize$0$}}
\put(15,0){\makebox(0,0){$\bullet$}}
\put(15,7){\makebox(0,0){\scriptsize$j$}}
\put(30,0){\makebox(0,0){$\bullet$}}
\put(30,7){\makebox(0,0){\scriptsize$0$}}
\put(60,0){\makebox(0,0){$\bullet$}}
\put(60,7){\makebox(0,0){\scriptsize$0$}}
\put(75,0){\makebox(0,0){$\bullet$}}
\put(75,7){\makebox(0,0){\scriptsize$0$}}
\put(67,0){\makebox(0,0){$\langle$}}
\end{picture}\;\raisebox{-2pt}{.}\]
For later use, we record the result that we shall require. 
\begin{corollary}\label{linearalgebra}
For $n\geq 2$, a smooth tensor on ${\mathbb{CP}}_n$ of the form
{\rm(\ref{generalisedRiemann})} vanishes upon restriction to every model
embedding ${\mathbb{RP}}_n\hookrightarrow{\mathbb{CP}}_n$ if and only if its
trace-free part with respect to the K\"ahler form vanishes.
\end{corollary}
\begin{proof} Immediate from Theorem~\ref{minortheorem} and our discussion 
above.\end{proof}
\section{Integrability conditions on ${\mathbb{CP}}_n$}\label{necessary}
In this section, we explore some necessary local conditions in order that a
smooth symmetric tensor $\omega_{ab\cdots c}$ on ${\mathbb{CP}}_n$ be of the
form~$\nabla_{(a}\phi_{b\cdots c)}$ for some symmetric~$\phi_{b\cdots c}$,
where $\nabla_a$ is the Fubini-Study connection. To do this, we shall need to
know the curvature of the Fubini-Study metric and also of the round metric
on~${\mathbb{RP}}_n$, as induced by a model embedding
${\mathbb{RP}}_n\hookrightarrow{\mathbb{CP}}_n$. With suitable normalisations,
the following formul\ae\ 
\begin{equation}\label{curvatures}\begin{array}{rcl}
\mbox{${\mathbb{RP}}_n$:}\quad R_{abcd}&\!=\!&
g_{ac}g_{bd}-g_{bc}g_{ad}\\
\mbox{${\mathbb{CP}}_n$:}\quad R_{abcd}&\!=\!&
g_{ac}g_{bd}-g_{bc}g_{ad}
+J_{ac}J_{bd}-J_{bc}J_{ad}+2J_{ab}J_{cd}
\end{array}\end{equation}
are well-known. Whilst the metric tensors $g_{ab}$ have different meanings on
the $2n$-dimensional Riemannian manifold ${\mathbb{CP}}_n$ and the
$n$-dimensional Riemannian manifold~${\mathbb{RP}}_n$, they coincide under
restriction to a model embedding
${\mathbb{RP}}_n\hookrightarrow{\mathbb{CP}}_n$ and, in this sense, our abuse
of notation is a legitimate convenience, which should cause no confusion.

The corresponding exploration on ${\mathbb{RP}}_n$ has already been done. To 
state its conclusions the following notation is useful. 
Suppose $T_{pq\cdots rab\cdots c}$ is a tensor that is symmetric in two 
groups of $\ell$ indices:--
\[T_{pq\cdots rab\cdots c}=T_{(pq\cdots r)(ab\cdots c)}.\]
We may define a new tensor by re-ordering its indices 
\[S_{paqb\cdots rc}=T_{pq\cdots rab\cdots c}\]
and then manufacturing yet another tensor by
\[R_{paqb\cdots rc}=S_{[pa][qb]\cdots [rc]}.\]
Let us write $\pi$ for this homomorphism of tensors
\[T_{pq\cdots rab\cdots c}\stackrel{\pi}{\longmapsto}R_{paqb\cdots rc}.\] 
Again, the notation applies equally well to tensors on ${\mathbb{RP}}_n$ as it 
does on~${\mathbb{CP}}_n$. Furthermore, it is evident that $\pi$ commutes with 
pull-back to a model embedding 
${\mathbb{RP}}_n\hookrightarrow{\mathbb{CP}}_n$. On the level of 
${\mathrm{SL}}(m,{\mathbb{R}})$-modules, where $m=2n$ for tensors on 
${\mathbb{CP}}_n$ and $m=n$ for tensors on~${\mathbb{RP}}_n$, 
the homomorphism $\pi$ is induced by projection onto the last factor of
\[\begin{picture}(60,10)
\put(0,0){\line(1,0){35}}
\put(45,0){\makebox(0,0){$\cdots$}}
\put(55,0){\line(1,0){5}}
\put(0,0){\makebox(0,0){$\bullet$}}
\put(0,7){\makebox(0,0){\scriptsize$\ell$}}
\put(15,0){\makebox(0,0){$\bullet$}}
\put(15,7){\makebox(0,0){\scriptsize$0$}}
\put(30,0){\makebox(0,0){$\bullet$}}
\put(30,7){\makebox(0,0){\scriptsize$0$}}
\put(60,0){\makebox(0,0){$\bullet$}}
\put(60,7){\makebox(0,0){\scriptsize$0$}}
\end{picture}\enskip\raisebox{-2pt}{$\otimes$}\enskip
\begin{picture}(60,10)
\put(0,0){\line(1,0){35}}
\put(45,0){\makebox(0,0){$\cdots$}}
\put(55,0){\line(1,0){5}}
\put(0,0){\makebox(0,0){$\bullet$}}
\put(0,7){\makebox(0,0){\scriptsize$\ell$}}
\put(15,0){\makebox(0,0){$\bullet$}}
\put(15,7){\makebox(0,0){\scriptsize$0$}}
\put(30,0){\makebox(0,0){$\bullet$}}
\put(30,7){\makebox(0,0){\scriptsize$0$}}
\put(60,0){\makebox(0,0){$\bullet$}}
\put(60,7){\makebox(0,0){\scriptsize$0$}}
\end{picture}\enskip\raisebox{-2pt}{$=\ldots\oplus$}\enskip
\begin{picture}(60,10)
\put(0,0){\line(1,0){35}}
\put(45,0){\makebox(0,0){$\cdots$}}
\put(55,0){\line(1,0){5}}
\put(0,0){\makebox(0,0){$\bullet$}}
\put(0,7){\makebox(0,0){\scriptsize$0$}}
\put(15,0){\makebox(0,0){$\bullet$}}
\put(15,7){\makebox(0,0){\scriptsize$\ell$}}
\put(30,0){\makebox(0,0){$\bullet$}}
\put(30,7){\makebox(0,0){\scriptsize$0$}}
\put(60,0){\makebox(0,0){$\bullet$}}
\put(60,7){\makebox(0,0){\scriptsize$0$}}
\end{picture}\quad(m-1\mbox{ nodes}),\]
the last module being trivial when $m=2$. On any manifold $M$ (but, in 
particular, on ${\mathbb{RP}}_n$ and~${\mathbb{CP}}_n$) let us introduce the
notation $Y^\ell$ for the bundle of $2\ell$-tensors $R_{paqb\cdots rc}$ 
satisfying the symmetries~(\ref{skewpairs}), symmetric in the pairs $pa$, 
$qb$, \dots, $rc$, and such that (\ref{bianchi}) holds. It is the 
target for the homomorphism
\[\textstyle\pi:\bigodot^\ell\!\Lambda^1\otimes\bigodot^\ell\!\Lambda^1
\longrightarrow Y^\ell,\]
which reduces to the exterior product 
$\wedge:\Lambda^1\otimes\Lambda^1\to\Lambda^2$ when $\ell=1$.

As will be explained in the proof, the following theorem may be derived from a 
special case of the {\em Bernstein-Gelfand-Gelfand resolution}. The cases
$\ell=2$ and $\ell=3$ were established by Calabi~\cite{c} (cf.~\cite{gg1}) 
and Estezet~\cite{e}, respectively.
\begin{theorem}\label{BGGonRP}
Fix $n\geq 2$. Suppose $\omega_{abc\cdots d}$ is a smooth
symmetric $\ell$-tensor on ${\mathbb{RP}}_n$. Then we can find a smooth 
symmetric tensor $\phi_{bc\cdots d}$ such that 
$\nabla_{(a}\phi_{bc\cdots d)}=\omega_{abc\cdots d}$ if and only if  
\begin{equation}\label{RPcompatibility}
\pi\left(\begin{array}l
\nabla_{(p}\nabla_q\nabla_r\cdots\nabla_{s)}\omega_{abc\cdots d}\\
\quad{}+{\textstyle\frac{(\ell-1)\ell(\ell+1)}{6}}
g_{(pq}\nabla_r\cdots\nabla_{s)}\omega_{abc\cdots d}\\
\qquad{}+\mbox{\rm lower order terms}\end{array}\right)=0.
\end{equation}
More explicitly, the operator in {\rm(\ref{RPcompatibility})} with its lower
order terms  may be determined as follows. If $\ell$ is even, 
then {\rm(\ref{RPcompatibility})} is
\begin{equation}\label{even}
\pi\big((\nabla^2+(\ell-1)^2 g)(\nabla^2+(\ell-3)^2 g)\cdots
(\nabla^2+9 g)(\nabla^2+ g)\big)\end{equation}
where
\[\textstyle\bigodot^{p-1}\!\Lambda^1\otimes\bigodot^\ell\!\Lambda^1
\xrightarrow{\nabla^2+p^2 g}
\bigodot^{p+1}\!\Lambda^1\otimes\bigodot^\ell\!\Lambda^1\]
is given by
\[\omega_{v \cdots wabc\cdots d}\mapsto
\nabla_{(t}\nabla_u\omega_{v\cdots w)abc\cdots d}+
p^2 g_{(tu}\omega_{v\cdots w)abc\cdots d}.\]
If $\ell$ is odd, then {\rm(\ref{RPcompatibility})} is
\[\pi\big((\nabla^2+(\ell-1)^2 g)(\nabla^2+(\ell-3)^2 g)\cdots
(\nabla^2+16 g)(\nabla^2+4 g)\nabla\big).\]
\end{theorem}
\begin{proof} Let us write $\nabla$ to stand for the operator 
$\phi_{bc\cdots d}\mapsto\nabla_{(a}\phi_{bc\cdots d)}$ and 
$\nabla^{(\ell)}$ for the differential operator 
in~(\ref{RPcompatibility}). We claim that the sequence
\begin{equation}\label{thesecondrow}
\textstyle\bigodot^{\ell-1}\!\Lambda^1\xrightarrow{\,\nabla\,}
\bigodot^{\ell}\!\Lambda^1\xrightarrow{\,\nabla^{(\ell)}\,}Y^\ell
\end{equation}
on the level of sheaves is part of a fine resolution of a certain locally 
constant sheaf on ${\mathbb{RP}}_n$. When $\ell=1$, we have in mind the 
de~Rham resolution
\[0\to{\mathbb{R}}\to\Lambda^0\xrightarrow{\nabla}\Lambda^1
\xrightarrow{\nabla}\Lambda^2\xrightarrow{\nabla}\Lambda^3
\xrightarrow{\nabla}\cdots\xrightarrow{\nabla}\Lambda^{n-2}
\xrightarrow{\nabla}\Lambda^{n-1}
\xrightarrow{\nabla}\Lambda^n\to 0\]
and Theorem~\ref{BGGonRP} follows because $H^1({\mathbb{RP}}_n,{\mathbb{R}})=0$
for $n\geq 2$. For $\ell\geq 2$, the BGG (Bernstein-Gelfand-Gelfand)
resolution~\cite{cd,cssiv} replaces de~Rham. The key point is that the round
metric on ${\mathbb{RP}}_n$ is projectively flat. As detailed
in~\cite{IMAnotes,eg}, the BGG resolution is
\[\begin{array}{l}\raisebox{-3pt}{$0\to$}\;
\framebox{\;\begin{picture}(100,10)
\put(0,0){\line(1,0){60}}
\put(70,0){\makebox(0,0){$\cdots$}}
\put(80,0){\line(1,0){20}}
\put(0,0){\makebox(0,0){$\bullet$}}
\put(0,7){\makebox(0,0){\scriptsize$0$}}
\put(20,0){\makebox(0,0){$\bullet$}}
\put(20,7){\makebox(0,0){\scriptsize$\ell-1$}}
\put(40,0){\makebox(0,0){$\bullet$}}
\put(40,7){\makebox(0,0){\scriptsize$0$}}
\put(55,0){\makebox(0,0){$\bullet$}}
\put(55,7){\makebox(0,0){\scriptsize$0$}}
\put(85,0){\makebox(0,0){$\bullet$}}
\put(85,7){\makebox(0,0){\scriptsize$0$}}
\put(100,0){\makebox(0,0){$\bullet$}}
\put(100,7){\makebox(0,0){\scriptsize$0$}}
\end{picture}\;}\\[10pt]
\quad\raisebox{-3pt}{$\to$}\enskip
\begin{picture}(80,10)
\put(0,0){\line(1,0){40}}
\put(50,0){\makebox(0,0){$\cdots$}}
\put(60,0){\line(1,0){20}}
\put(0,0){\makebox(0,0){$\bullet$}}
\put(0,7){\makebox(0,0){\scriptsize$\ell-1$}}
\put(20,0){\makebox(0,0){$\bullet$}}
\put(20,7){\makebox(0,0){\scriptsize$0$}}
\put(35,0){\makebox(0,0){$\bullet$}}
\put(35,7){\makebox(0,0){\scriptsize$0$}}
\put(65,0){\makebox(0,0){$\bullet$}}
\put(65,7){\makebox(0,0){\scriptsize$0$}}
\put(80,0){\makebox(0,0){$\bullet$}}
\put(80,7){\makebox(0,0){\scriptsize$0$}}
\end{picture}
\enskip\raisebox{-3pt}{$\xrightarrow{\,\nabla\,}$}\enskip
\begin{picture}(75,10)
\put(0,0){\line(1,0){35}}
\put(45,0){\makebox(0,0){$\cdots$}}
\put(55,0){\line(1,0){20}}
\put(0,0){\makebox(0,0){$\bullet$}}
\put(0,7){\makebox(0,0){\scriptsize$\ell$}}
\put(15,0){\makebox(0,0){$\bullet$}}
\put(15,7){\makebox(0,0){\scriptsize$0$}}
\put(30,0){\makebox(0,0){$\bullet$}}
\put(30,7){\makebox(0,0){\scriptsize$0$}}
\put(60,0){\makebox(0,0){$\bullet$}}
\put(60,7){\makebox(0,0){\scriptsize$0$}}
\put(75,0){\makebox(0,0){$\bullet$}}
\put(75,7){\makebox(0,0){\scriptsize$0$}}
\end{picture}
\enskip\raisebox{-3pt}{$\xrightarrow{\,\nabla^{(\ell)}\,}$}\enskip
\begin{picture}(75,10)
\put(0,0){\line(1,0){35}}
\put(45,0){\makebox(0,0){$\cdots$}}
\put(55,0){\line(1,0){20}}
\put(0,0){\makebox(0,0){$\bullet$}}
\put(0,7){\makebox(0,0){\scriptsize$0$}}
\put(15,0){\makebox(0,0){$\bullet$}}
\put(15,7){\makebox(0,0){\scriptsize$\ell$}}
\put(30,0){\makebox(0,0){$\bullet$}}
\put(30,7){\makebox(0,0){\scriptsize$0$}}
\put(60,0){\makebox(0,0){$\bullet$}}
\put(60,7){\makebox(0,0){\scriptsize$0$}}
\put(75,0){\makebox(0,0){$\bullet$}}
\put(75,7){\makebox(0,0){\scriptsize$0$}}
\end{picture}\\[10pt]
\qquad\raisebox{-3pt}{$\xrightarrow{\,\nabla\,}$}\enskip
\begin{picture}(85,10)
\put(0,0){\line(1,0){45}}
\put(55,0){\makebox(0,0){$\cdots$}}
\put(65,0){\line(1,0){20}}
\put(0,0){\makebox(0,0){$\bullet$}}
\put(0,7){\makebox(0,0){\scriptsize$0$}}
\put(20,0){\makebox(0,0){$\bullet$}}
\put(20,7){\makebox(0,0){\scriptsize$\ell-1$}}
\put(40,0){\makebox(0,0){$\bullet$}}
\put(40,7){\makebox(0,0){\scriptsize$1$}}
\put(70,0){\makebox(0,0){$\bullet$}}
\put(70,7){\makebox(0,0){\scriptsize$0$}}
\put(85,0){\makebox(0,0){$\bullet$}}
\put(85,7){\makebox(0,0){\scriptsize$0$}}
\end{picture}
\enskip\raisebox{-3pt}{$\xrightarrow{\,\nabla\,}\cdots
\xrightarrow{\,\nabla\,}$}\enskip
\begin{picture}(85,10)
\put(0,0){\line(1,0){45}}
\put(55,0){\makebox(0,0){$\cdots$}}
\put(65,0){\line(1,0){20}}
\put(0,0){\makebox(0,0){$\bullet$}}
\put(0,7){\makebox(0,0){\scriptsize$0$}}
\put(20,0){\makebox(0,0){$\bullet$}}
\put(20,7){\makebox(0,0){\scriptsize$\ell-1$}}
\put(40,0){\makebox(0,0){$\bullet$}}
\put(40,7){\makebox(0,0){\scriptsize$0$}}
\put(70,0){\makebox(0,0){$\bullet$}}
\put(70,7){\makebox(0,0){\scriptsize$1$}}
\put(85,0){\makebox(0,0){$\bullet$}}
\put(85,7){\makebox(0,0){\scriptsize$0$}}
\end{picture}\\[10pt]
\quad\qquad\raisebox{-3pt}{$\xrightarrow{\,\nabla\,}$}\enskip
\begin{picture}(85,10)
\put(0,0){\line(1,0){45}}
\put(55,0){\makebox(0,0){$\cdots$}}
\put(65,0){\line(1,0){20}}
\put(0,0){\makebox(0,0){$\bullet$}}
\put(0,7){\makebox(0,0){\scriptsize$0$}}
\put(20,0){\makebox(0,0){$\bullet$}}
\put(20,7){\makebox(0,0){\scriptsize$\ell-1$}}
\put(40,0){\makebox(0,0){$\bullet$}}
\put(40,7){\makebox(0,0){\scriptsize$0$}}
\put(70,0){\makebox(0,0){$\bullet$}}
\put(70,7){\makebox(0,0){\scriptsize$0$}}
\put(85,0){\makebox(0,0){$\bullet$}}
\put(85,7){\makebox(0,0){\scriptsize$1$}}
\end{picture}
\enskip\raisebox{-3pt}{$\xrightarrow{\,\nabla\,}$}\enskip
\begin{picture}(85,10)
\put(0,0){\line(1,0){45}}
\put(55,0){\makebox(0,0){$\cdots$}}
\put(65,0){\line(1,0){20}}
\put(0,0){\makebox(0,0){$\bullet$}}
\put(0,7){\makebox(0,0){\scriptsize$0$}}
\put(20,0){\makebox(0,0){$\bullet$}}
\put(20,7){\makebox(0,0){\scriptsize$\ell-1$}}
\put(40,0){\makebox(0,0){$\bullet$}}
\put(40,7){\makebox(0,0){\scriptsize$0$}}
\put(70,0){\makebox(0,0){$\bullet$}}
\put(70,7){\makebox(0,0){\scriptsize$0$}}
\put(85,0){\makebox(0,0){$\bullet$}}
\put(85,7){\makebox(0,0){\scriptsize$0$}}
\end{picture}
\enskip\raisebox{-2pt}{.}\end{array}\]
Here, \raisebox{3pt}{\framebox{\rule{0pt}{3pt}\quad}} denotes the locally
constant sheaf on ${\mathbb{RP}}_n={\mathrm{SL}}(n+1,{\mathbb{R}})/P$ induced
by the given representation of ${\mathrm{SL}}(n+1,{\mathbb{R}})$ restricted to
$P$ and then induced back up to a homogeneous bundle on~${\mathbb{RP}}_n$. 
After \raisebox{3pt}{\framebox{\rule{0pt}{3pt}\quad}}, we are 
writing irreducible representations of
${\mathrm{SL}}(n,{\mathbb{R}})$ instead of the corresponding induced bundles.
The second row coincides with~(\ref{thesecondrow}). Formul\ae\ for the
operators $\nabla^{(\ell)}$ are given inductively in~\cite{cds,cssiii} but the  
products given in the statement of Theorem~\ref{BGGonRP} are most easily 
deduced from Gover's method~\cite{gover} (and these factorisations hold for 
any Einstein metric). In expanded form
\[\pi\Big(\textstyle\nabla^{\ell}\omega
+\frac{(\ell-1)\ell(\ell+1)}{6}g\nabla^{\ell-2}\omega
+\frac{(\ell-3)(\ell-2)(\ell-1)\ell(\ell+1)(5\ell+7)}{360}
g^2\nabla^{\ell-4}\omega+\cdots
\Big)\]
the coefficients are quite complicated (although the array generated by 
(\ref{even}) appears as triangle A008956 in the On-Line Encyclopedia of 
Integer Sequences at \verb+www.research.att.com/~njas/sequences+). Fortunately, 
we shall not need the details of the operators $\nabla^{(\ell)}$ but only 
their general form and how they may be manufactured, which is as follows. Let 
${\mathbb{T}}$ denote the bundle $\Lambda^0\oplus\Lambda^1$ on 
${\mathbb{RP}}_n$ equipped with the connection
\begin{equation}\label{standardtractorconnection}
{\mathbb{T}}=\begin{array}c\Lambda^0\\ \oplus\\ \Lambda^1\end{array}\ni
\left[\begin{array}c\sigma\\ \mu_b\end{array}\right]
\stackrel{\nabla_a}{\longmapsto}
\left[\begin{array}c\nabla_a\sigma-\mu_a\\ 
\nabla_a\mu_b+g_{ab}\sigma\end{array}\right]\in\Lambda^1\otimes{\mathbb{T}}.
\end{equation}
We compute that
\[\nabla_a\nabla_b
\left[\begin{array}c\sigma\\ \mu_c\end{array}\right]=
\left[\begin{array}c\nabla_a\nabla_b\sigma-\nabla_a\mu_b
-\nabla_b\mu_a-g_{ab}\sigma\\ 
\nabla_a\nabla_b\mu_c+g_{bc}\nabla_a\sigma+g_{ac}\nabla_b\sigma
-g_{ac}\mu_b\end{array}\right]\]
and observe from (\ref{curvatures}) that
\[(\nabla_a\nabla_b-\nabla_b\nabla_a)\mu_c=R_{abc}{}^d\mu_d=
g_{ac}\mu_b-g_{bc}\mu_a\,,\]
hence that the connection on ${\mathbb{T}}$ is flat. It follows that the 
coupled de~Rham complex 
\[\Lambda^0\otimes{\mathbb{T}}\xrightarrow{\,\nabla\,}
\Lambda^1\otimes{\mathbb{T}}\xrightarrow{\,\nabla\,}
\Lambda^2\otimes{\mathbb{T}}\xrightarrow{\,\nabla\,}
\Lambda^3\otimes{\mathbb{T}}\xrightarrow{\,\nabla\,}
\cdots\xrightarrow{\,\nabla\,}\Lambda^n\otimes{\mathbb{T}}\to 0\]
is a fine resolution of the locally covariant constant sections 
of~${\mathbb{T}}$ with a similar conclusion for any associated vector bundle 
such as $\Lambda^2{\mathbb{T}}$. Let us examine the induced connection on 
$\Lambda^2{\mathbb{T}}$ in more detail:--
\begin{equation}\label{skewtractors}
\Lambda^2{\mathbb{T}}=\!\begin{array}c\Lambda^1\\ \oplus\\ \Lambda^2\end{array}
\!\ni\left[\begin{array}c\sigma_b\\ \mu_{bc}\end{array}\right]
\stackrel{\nabla_a}{\longmapsto}
\left[\begin{array}c\nabla_a\sigma_b-\mu_{ab}\\ 
\nabla_a\mu_{bc}+g_{ab}\sigma_c-g_{ac}\sigma_b\end{array}\right]
\in\Lambda^1\otimes\Lambda^2{\mathbb{T}}.\end{equation}
We shall show that exactness of the coupled de~Rham complex 
\[\Gamma({\mathbb{RP}}_n,\Lambda^2{\mathbb{T}})\xrightarrow{\,\nabla\,}
\Gamma({\mathbb{RP}}_n,\Lambda^1\otimes\Lambda^2{\mathbb{T}})
\xrightarrow{\,\nabla\,}
\Gamma({\mathbb{RP}}_n,\Lambda^2\otimes\Lambda^2{\mathbb{T}})\]
is sufficient to deduce the case $\ell=2$ of Theorem~\ref{BGGonRP}. 
Firstly, we need a formula for 
$\nabla_a:\Lambda^1\otimes\Lambda^2{\mathbb{T}}\to
\Lambda^2\otimes\Lambda^2{\mathbb{T}}$. It is
immediate from (\ref{skewtractors}) that
\begin{equation}\label{coupledskewtractors}
\left[\begin{array}c\xi_{bc}\\ \nu_{bcd}\end{array}\right]
\stackrel{\nabla_a}{\longmapsto}
\left[\begin{array}c\nabla_{[a}\xi_{b]c}+\nu_{[ab]c}\\ 
\nabla_{[a}\nu_{b]cd}+g_{c[a}\xi_{b]d}-g_{d[a}\xi_{b]c}\end{array}\right].
\end{equation}
Now suppose 
$\omega_{ab}$ is a symmetric tensor on~${\mathbb{RP}}_n$. We claim 
that the following are equivalent.
\begin{enumerate}\setlength{\itemsep}{3pt}
\item $\left[\begin{array}c\omega_{bc}\\ 
\nabla_c\omega_{db}-\nabla_d\omega_{cb}\end{array}\right]
\in\Gamma({\mathbb{RP}}_n,\Lambda^1\otimes\Lambda^2{\mathbb{T}})$
is in the range of the coupled connection 
$\nabla_b:\Gamma({\mathbb{RP}}_n,\Lambda^2{\mathbb{T}})\to
\Gamma({\mathbb{RP}}_n,\Lambda^1\otimes\Lambda^2{\mathbb{T}})$.
\item $\omega_{ab}=\nabla_{(a}\phi_{b)}$ for some 
$\phi_a\in\Gamma({\mathbb{RP}}_n,\Lambda^1)$.
\item $\left[\begin{array}c\omega_{bc}\\ 
\nu_{bcd}\end{array}\right]
\in\Gamma({\mathbb{RP}}_n,\Lambda^1\otimes\Lambda^2{\mathbb{T}})$
for some $\nu_{bcd}\in\Gamma({\mathbb{RP}}_n,\Lambda^1\otimes\Lambda^2)$ 
is in the range of 
$\nabla_b:\Gamma({\mathbb{RP}}_n,\Lambda^2{\mathbb{T}})\to
\Gamma({\mathbb{RP}}_n,\Lambda^1\otimes\Lambda^2{\mathbb{T}})$.
\end{enumerate}
It is clear from (\ref{skewtractors}) that
(i)$\Rightarrow$(ii)$\Rightarrow$(iii). It remains to show 
(iii)$\Rightarrow$(i). To see this, recall that the curvature of the 
connection on $\Lambda^2{\mathbb{T}}$ is flat and so if (iii) holds, then we 
must have
\[\left[\begin{array}c\omega_{bc}\\ \nu_{bcd}\end{array}\right]
\stackrel{\nabla_a}{\longmapsto}0\in
\Gamma({\mathbb{RP}}_n,\Lambda^2\otimes\Lambda^2{\mathbb{T}}).\]
In particular, we read off from the first row of (\ref{coupledskewtractors}) 
that 
\[\nabla_{[a}\omega_{b]c}+\nu_{[ab]c}=0.\]
{From} this, bearing in mind that $\nu_{abc}=\nu_{a[bc]}$, it follows that 
\[\nu_{abc}=3\nu_{[abc]}-2\nu_{[bc]a}=2\nabla_{[b}\omega_{c]a}=
\nabla_{b}\omega_{ca}-\nabla_{c}\omega_{ba},\]
as required. Finally, to deduce Theorem~\ref{BGGonRP} we suppose that (i) holds
and consider the second row of~(\ref{coupledskewtractors}):--
\[\left[\!\begin{array}c\omega_{bc}\\ 
\nabla_c\omega_{db}-\nabla_d\omega_{cb}\end{array}\!\right]
\!\stackrel{\nabla_a}{\longmapsto}\!
\left[\!\begin{array}c0\\ 
\nabla_{[a}\nabla_{|c|}\omega_{b]d}-\nabla_{[a}\nabla_{|d|}\omega_{b]c}
+g_{c[a}\omega_{b]d}-g_{d[a}\omega_{b]c}\end{array}\!\right]
\]
where, following~\cite[pp.~132--]{PR1}, the vertical bars in $\nabla_{|c|}$ and
$\nabla_{|d|}$ exclude the indices they enclose from skew-symmetrisation. 
Again, since the connection is flat, we conclude that the vanishing of 
\begin{equation}\label{rawobstruction}
\nabla_{[a}\nabla_{|c|}\omega_{b]d}-\nabla_{[a}\nabla_{|d|}\omega_{b]c}
+g_{c[a}\omega_{b]d}-g_{d[a}\omega_{b]c}\end{equation}
is a necessary and sufficient condition in order that 
$\omega_{ab}=\nabla_{(a}\phi_{b)}$ for some smooth $1$-form $\phi_b$ 
on~${\mathbb{RP}}_n$. However, it is easy to check that this coincides with  
\begin{equation}\label{refinedobstruction}
2\times\pi(\nabla_{(a}\nabla_{c)}\omega_{bd}+g_{ac}\omega_{bd})\end{equation}
whence Theorem~\ref{BGGonRP} is proved for the case $\ell=2$. The general case 
follows similarly by considering the induced flat connection on the associated
bundle (in terms of Young diagrams)
\begin{equation}\label{bigtractors}\underbrace{\begin{picture}(54,12)
\put(0,0){\line(1,0){54}}
\put(0,6){\line(1,0){54}}
\put(0,12){\line(1,0){54}}
\put(0,0){\line(0,1){12}}
\put(6,0){\line(0,1){12}}
\put(12,0){\line(0,1){12}}
\put(18,0){\line(0,1){12}}
\put(30,3){\makebox(0,0){$\cdots$}}
\put(30,9){\makebox(0,0){$\cdots$}}
\put(42,0){\line(0,1){12}}
\put(48,0){\line(0,1){12}}
\put(54,0){\line(0,1){12}}
\end{picture}}_{\mbox{\scriptsize$\ell-1$ columns}}\,
\mbox{\LARGE${\mathbb{T}}$}\end{equation}
or equivalently (since $\Lambda^n$ and hence 
$\Lambda^{n+1}{\mathbb{T}}=\Lambda^0\otimes\Lambda^n$ are trivialised by 
the round volume form), the bundle induced from the special frame-bundle of 
${\mathbb{T}}$ by the representation 
\begin{equation}\label{heavytractors}\begin{picture}(100,10)
\put(0,0){\line(1,0){60}}
\put(70,0){\makebox(0,0){$\cdots$}}
\put(80,0){\line(1,0){20}}
\put(0,0){\makebox(0,0){$\bullet$}}
\put(0,7){\makebox(0,0){\scriptsize$0$}}
\put(20,0){\makebox(0,0){$\bullet$}}
\put(20,7){\makebox(0,0){\scriptsize$\ell-1$}}
\put(40,0){\makebox(0,0){$\bullet$}}
\put(40,7){\makebox(0,0){\scriptsize$0$}}
\put(55,0){\makebox(0,0){$\bullet$}}
\put(55,7){\makebox(0,0){\scriptsize$0$}}
\put(85,0){\makebox(0,0){$\bullet$}}
\put(85,7){\makebox(0,0){\scriptsize$0$}}
\put(100,0){\makebox(0,0){$\bullet$}}
\put(100,7){\makebox(0,0){\scriptsize$0$}}
\end{picture}\end{equation}
of ${\mathrm{SL}}(n+1,{\mathbb{R}})$. A salient feature of this bundle is 
its structure when written as ordinary tensor bundles on~${\mathbb{RP}}_n$:--
\[\begin{array}l\begin{picture}(54,12)
\put(0,0){\line(1,0){54}}
\put(0,6){\line(1,0){54}}
\put(0,12){\line(1,0){54}}
\put(0,0){\line(0,1){12}}
\put(6,0){\line(0,1){12}}
\put(12,0){\line(0,1){12}}
\put(18,0){\line(0,1){12}}
\put(30,3){\makebox(0,0){$\cdots$}}
\put(30,9){\makebox(0,0){$\cdots$}}
\put(42,0){\line(0,1){12}}
\put(48,0){\line(0,1){12}}
\put(54,0){\line(0,1){12}}
\end{picture}\;
\mbox{\LARGE${\mathbb{T}}$}\raisebox{2pt}{ $=\quad
\begin{picture}(80,10)
\put(0,0){\line(1,0){40}}
\put(50,0){\makebox(0,0){$\cdots$}}
\put(60,0){\line(1,0){20}}
\put(0,0){\makebox(0,0){$\bullet$}}
\put(0,7){\makebox(0,0){\scriptsize$\ell-1$}}
\put(20,0){\makebox(0,0){$\bullet$}}
\put(20,7){\makebox(0,0){\scriptsize$0$}}
\put(35,0){\makebox(0,0){$\bullet$}}
\put(35,7){\makebox(0,0){\scriptsize$0$}}
\put(65,0){\makebox(0,0){$\bullet$}}
\put(65,7){\makebox(0,0){\scriptsize$0$}}
\put(80,0){\makebox(0,0){$\bullet$}}
\put(80,7){\makebox(0,0){\scriptsize$0$}}
\end{picture}\enskip\oplus\quad\begin{picture}(80,10)
\put(0,0){\line(1,0){40}}
\put(50,0){\makebox(0,0){$\cdots$}}
\put(60,0){\line(1,0){20}}
\put(0,0){\makebox(0,0){$\bullet$}}
\put(0,7){\makebox(0,0){\scriptsize$\ell-2$}}
\put(20,0){\makebox(0,0){$\bullet$}}
\put(20,7){\makebox(0,0){\scriptsize$1$}}
\put(35,0){\makebox(0,0){$\bullet$}}
\put(35,7){\makebox(0,0){\scriptsize$0$}}
\put(65,0){\makebox(0,0){$\bullet$}}
\put(65,7){\makebox(0,0){\scriptsize$0$}}
\put(80,0){\makebox(0,0){$\bullet$}}
\put(80,7){\makebox(0,0){\scriptsize$0$}}
\end{picture}$}\\
\hspace*{80pt}\oplus\quad\begin{picture}(80,10)
\put(0,0){\line(1,0){40}}
\put(50,0){\makebox(0,0){$\cdots$}}
\put(60,0){\line(1,0){20}}
\put(0,0){\makebox(0,0){$\bullet$}}
\put(0,7){\makebox(0,0){\scriptsize$\ell-3$}}
\put(20,0){\makebox(0,0){$\bullet$}}
\put(20,7){\makebox(0,0){\scriptsize$2$}}
\put(35,0){\makebox(0,0){$\bullet$}}
\put(35,7){\makebox(0,0){\scriptsize$0$}}
\put(65,0){\makebox(0,0){$\bullet$}}
\put(65,7){\makebox(0,0){\scriptsize$0$}}
\put(80,0){\makebox(0,0){$\bullet$}}
\put(80,7){\makebox(0,0){\scriptsize$0$}}
\end{picture}\enskip\oplus\cdots\oplus\quad
\begin{picture}(85,10)
\put(0,0){\line(1,0){45}}
\put(55,0){\makebox(0,0){$\cdots$}}
\put(65,0){\line(1,0){20}}
\put(0,0){\makebox(0,0){$\bullet$}}
\put(0,7){\makebox(0,0){\scriptsize$0$}}
\put(20,0){\makebox(0,0){$\bullet$}}
\put(20,7){\makebox(0,0){\scriptsize$\ell-1$}}
\put(40,0){\makebox(0,0){$\bullet$}}
\put(40,7){\makebox(0,0){\scriptsize$0$}}
\put(70,0){\makebox(0,0){$\bullet$}}
\put(70,7){\makebox(0,0){\scriptsize$0$}}
\put(85,0){\makebox(0,0){$\bullet$}}
\put(85,7){\makebox(0,0){\scriptsize$0$}}
\end{picture}\end{array}\]
in terms of which the induced connection takes the form
\begin{equation}\label{heavyconnection}
\left[\begin{array}c\sigma\\ \mu\\ \rho\\ \vdots\end{array}\right]
\stackrel{\nabla}{\longmapsto}
\left[\begin{array}c\nabla\sigma-\partial\mu\\ 
\nabla\mu-\partial\rho+g\bowtie\sigma\\ \vdots \\ {}
\end{array}\right]\end{equation}
for some appropriate tensor combination $g\bowtie\sigma$ where, 
following~\cite{bceg}, the homomorphism 
\[\raisebox{3pt}{$\partial:{}$}\begin{picture}(54,12)
\put(0,0){\line(1,0){54}}
\put(0,6){\line(1,0){54}}
\put(0,12){\line(1,0){54}}
\put(0,0){\line(0,1){12}}
\put(6,0){\line(0,1){12}}
\put(12,0){\line(0,1){12}}
\put(18,0){\line(0,1){12}}
\put(30,3){\makebox(0,0){$\cdots$}}
\put(30,9){\makebox(0,0){$\cdots$}}
\put(42,0){\line(0,1){12}}
\put(48,0){\line(0,1){12}}
\put(54,0){\line(0,1){12}}
\end{picture}\;
\mbox{\LARGE${\mathbb{T}}$}\raisebox{3pt}{ $\longrightarrow\Lambda^1\otimes$ }
\begin{picture}(54,12)
\put(0,0){\line(1,0){54}}
\put(0,6){\line(1,0){54}}
\put(0,12){\line(1,0){54}}
\put(0,0){\line(0,1){12}}
\put(6,0){\line(0,1){12}}
\put(12,0){\line(0,1){12}}
\put(18,0){\line(0,1){12}}
\put(30,3){\makebox(0,0){$\cdots$}}
\put(30,9){\makebox(0,0){$\cdots$}}
\put(42,0){\line(0,1){12}}
\put(48,0){\line(0,1){12}}
\put(54,0){\line(0,1){12}}
\end{picture}\;
\mbox{\LARGE${\mathbb{T}}$}
\]
is best regarded as induced by the Lie algebra differential 
\begin{equation}\label{partial}\partial:\enskip\begin{picture}(100,10)
\put(0,0){\line(1,0){60}}
\put(70,0){\makebox(0,0){$\cdots$}}
\put(80,0){\line(1,0){20}}
\put(0,0){\makebox(0,0){$\bullet$}}
\put(0,7){\makebox(0,0){\scriptsize$0$}}
\put(20,0){\makebox(0,0){$\bullet$}}
\put(20,7){\makebox(0,0){\scriptsize$\ell-1$}}
\put(40,0){\makebox(0,0){$\bullet$}}
\put(40,7){\makebox(0,0){\scriptsize$0$}}
\put(55,0){\makebox(0,0){$\bullet$}}
\put(55,7){\makebox(0,0){\scriptsize$0$}}
\put(85,0){\makebox(0,0){$\bullet$}}
\put(85,7){\makebox(0,0){\scriptsize$0$}}
\put(100,0){\makebox(0,0){$\bullet$}}
\put(100,7){\makebox(0,0){\scriptsize$0$}}
\end{picture}\enskip\longrightarrow
{\mathfrak{g}}_1\otimes\enskip\begin{picture}(100,10)
\put(0,0){\line(1,0){60}}
\put(70,0){\makebox(0,0){$\cdots$}}
\put(80,0){\line(1,0){20}}
\put(0,0){\makebox(0,0){$\bullet$}}
\put(0,7){\makebox(0,0){\scriptsize$0$}}
\put(20,0){\makebox(0,0){$\bullet$}}
\put(20,7){\makebox(0,0){\scriptsize$\ell-1$}}
\put(40,0){\makebox(0,0){$\bullet$}}
\put(40,7){\makebox(0,0){\scriptsize$0$}}
\put(55,0){\makebox(0,0){$\bullet$}}
\put(55,7){\makebox(0,0){\scriptsize$0$}}
\put(85,0){\makebox(0,0){$\bullet$}}
\put(85,7){\makebox(0,0){\scriptsize$0$}}
\put(100,0){\makebox(0,0){$\bullet$}}
\put(100,7){\makebox(0,0){\scriptsize$0$}}
\end{picture}\:\raisebox{-2pt}{.}\end{equation}
Here, ${\mathfrak{sl}}(n+1,{\mathbb{R}})$ is regarded as a $|1|$-graded Lie 
algebra
\[\begin{array}l{\mathfrak{sl}}(n+1,{\mathbb{R}})=
{\mathfrak{g}}_{-1}\oplus{\mathfrak{g}}\oplus{\mathfrak{g}}_1\\
=\mbox{\footnotesize$
\left\{
\left\lgroup\begin{tabular}{c|ccc}$0$&$0$&$\!\cdots\!$&$0$\\ \hline
$\ast$\\
\raisebox{2pt}[15pt]{$\vdots$}&&
\raisebox{-2pt}[0pt][0pt]{\makebox[0pt]{\Huge$0$}}\\
$\ast$\end{tabular}\right\rgroup\right\}\mbox{\normalsize$\oplus$}
\left\{
\left\lgroup\begin{tabular}{c|ccc}$\ast$&$0$&$\!\cdots\!$&$0$\\ \hline
$0$\\
\raisebox{2pt}[15pt]{$\vdots$}&&
\raisebox{-2pt}[0pt][0pt]{\makebox[0pt]{\Huge$\ast$}}\\
$0$\end{tabular}\right\rgroup\right\}\mbox{\normalsize$\oplus$}
\left\{
\left\lgroup\begin{tabular}{c|ccc}$0$&$\ast$&$\!\cdots\!$&$\ast$\\ \hline
$0$\\
\raisebox{2pt}[15pt]{$\vdots$}&&
\raisebox{-2pt}[0pt][0pt]{\makebox[0pt]{\Huge$0$}}\\
$0$\end{tabular}\right\rgroup\right\}$}\end{array}\]
and the ${\mathfrak{sl}}(n+1,{\mathbb{R}})$-module (\ref{heavytractors}) is 
restricted to ${\mathfrak{g}}_{-1}$ for the purposes of~(\ref{partial}). The 
operator 
$\phi_{bc\cdots d}\mapsto\omega_{abc\cdots d}=\nabla_{(a}\phi_{bc\cdots d)}$
appears within (\ref{heavyconnection}) as the equation
$\nabla\sigma-\partial\mu=\omega$ for some~$\mu$ and the analogues of (i), 
(ii), (iii) from the case $\ell=2$ readily arise. As detailed in~\cite{bceg}, 
the Lie algebra cohomologies
\[\begin{array}{rcl}H^0({\mathfrak{g}}_{-1},\;\begin{picture}(100,10)
\put(0,0){\line(1,0){60}}
\put(70,0){\makebox(0,0){$\cdots$}}
\put(80,0){\line(1,0){20}}
\put(0,0){\makebox(0,0){$\bullet$}}
\put(0,7){\makebox(0,0){\scriptsize$0$}}
\put(20,0){\makebox(0,0){$\bullet$}}
\put(20,7){\makebox(0,0){\scriptsize$\ell-1$}}
\put(40,0){\makebox(0,0){$\bullet$}}
\put(40,7){\makebox(0,0){\scriptsize$0$}}
\put(55,0){\makebox(0,0){$\bullet$}}
\put(55,7){\makebox(0,0){\scriptsize$0$}}
\put(85,0){\makebox(0,0){$\bullet$}}
\put(85,7){\makebox(0,0){\scriptsize$0$}}
\put(100,0){\makebox(0,0){$\bullet$}}
\put(100,7){\makebox(0,0){\scriptsize$0$}}
\end{picture}\;)&=&\enskip\begin{picture}(80,10)
\put(0,0){\line(1,0){40}}
\put(50,0){\makebox(0,0){$\cdots$}}
\put(60,0){\line(1,0){20}}
\put(0,0){\makebox(0,0){$\bullet$}}
\put(0,7){\makebox(0,0){\scriptsize$\ell-1$}}
\put(20,0){\makebox(0,0){$\bullet$}}
\put(20,7){\makebox(0,0){\scriptsize$0$}}
\put(35,0){\makebox(0,0){$\bullet$}}
\put(35,7){\makebox(0,0){\scriptsize$0$}}
\put(65,0){\makebox(0,0){$\bullet$}}
\put(65,7){\makebox(0,0){\scriptsize$0$}}
\put(80,0){\makebox(0,0){$\bullet$}}
\put(80,7){\makebox(0,0){\scriptsize$0$}}
\end{picture}\\[5pt]
H^1({\mathfrak{g}}_{-1},\;\begin{picture}(100,10)
\put(0,0){\line(1,0){60}}
\put(70,0){\makebox(0,0){$\cdots$}}
\put(80,0){\line(1,0){20}}
\put(0,0){\makebox(0,0){$\bullet$}}
\put(0,7){\makebox(0,0){\scriptsize$0$}}
\put(20,0){\makebox(0,0){$\bullet$}}
\put(20,7){\makebox(0,0){\scriptsize$\ell-1$}}
\put(40,0){\makebox(0,0){$\bullet$}}
\put(40,7){\makebox(0,0){\scriptsize$0$}}
\put(55,0){\makebox(0,0){$\bullet$}}
\put(55,7){\makebox(0,0){\scriptsize$0$}}
\put(85,0){\makebox(0,0){$\bullet$}}
\put(85,7){\makebox(0,0){\scriptsize$0$}}
\put(100,0){\makebox(0,0){$\bullet$}}
\put(100,7){\makebox(0,0){\scriptsize$0$}}
\end{picture}\;)&=&\begin{picture}(75,10)
\put(0,0){\line(1,0){35}}
\put(45,0){\makebox(0,0){$\cdots$}}
\put(55,0){\line(1,0){20}}
\put(0,0){\makebox(0,0){$\bullet$}}
\put(0,7){\makebox(0,0){\scriptsize$\ell$}}
\put(15,0){\makebox(0,0){$\bullet$}}
\put(15,7){\makebox(0,0){\scriptsize$0$}}
\put(30,0){\makebox(0,0){$\bullet$}}
\put(30,7){\makebox(0,0){\scriptsize$0$}}
\put(60,0){\makebox(0,0){$\bullet$}}
\put(60,7){\makebox(0,0){\scriptsize$0$}}
\put(75,0){\makebox(0,0){$\bullet$}}
\put(75,7){\makebox(0,0){\scriptsize$0$}}
\end{picture}
\end{array}\]
provide the representations of ${\mathrm{SL}}(n,{\mathbb{R}})$ inducing the
vector bundles on ${\mathbb{RP}}_n$ between which the differential operator
$\phi_{bc\cdots d}\mapsto\nabla_{(a}\phi_{bc\cdots d)}$ acts. Reasoning as 
for the case $\ell=2$ above, it is the second cohomology 
\[H^2({\mathfrak{g}}_{-1},\;\begin{picture}(100,10)
\put(0,0){\line(1,0){60}}
\put(70,0){\makebox(0,0){$\cdots$}}
\put(80,0){\line(1,0){20}}
\put(0,0){\makebox(0,0){$\bullet$}}
\put(0,7){\makebox(0,0){\scriptsize$0$}}
\put(20,0){\makebox(0,0){$\bullet$}}
\put(20,7){\makebox(0,0){\scriptsize$\ell-1$}}
\put(40,0){\makebox(0,0){$\bullet$}}
\put(40,7){\makebox(0,0){\scriptsize$0$}}
\put(55,0){\makebox(0,0){$\bullet$}}
\put(55,7){\makebox(0,0){\scriptsize$0$}}
\put(85,0){\makebox(0,0){$\bullet$}}
\put(85,7){\makebox(0,0){\scriptsize$0$}}
\put(100,0){\makebox(0,0){$\bullet$}}
\put(100,7){\makebox(0,0){\scriptsize$0$}}
\end{picture}\;)\enskip=\enskip\begin{picture}(75,10)
\put(0,0){\line(1,0){35}}
\put(45,0){\makebox(0,0){$\cdots$}}
\put(55,0){\line(1,0){20}}
\put(0,0){\makebox(0,0){$\bullet$}}
\put(0,7){\makebox(0,0){\scriptsize$0$}}
\put(15,0){\makebox(0,0){$\bullet$}}
\put(15,7){\makebox(0,0){\scriptsize$\ell$}}
\put(30,0){\makebox(0,0){$\bullet$}}
\put(30,7){\makebox(0,0){\scriptsize$0$}}
\put(60,0){\makebox(0,0){$\bullet$}}
\put(60,7){\makebox(0,0){\scriptsize$0$}}
\put(75,0){\makebox(0,0){$\bullet$}}
\put(75,7){\makebox(0,0){\scriptsize$0$}}
\end{picture}
\]
(computed according to Kostant's Theorem~\cite{k}) that induces the vector
bundle providing the obstruction to being in the range of this operator.

For the purposes of this article, it is not necessary to know the exact formula
for this obstruction but only how it arises from (\ref{bigtractors}) with its 
flat connection and that it has the form
\[\pi\left(
\nabla_{(p}\nabla_q\nabla_r\cdots\nabla_{s)}\omega_{abc\cdots d}+
\mbox{\rm lower order $g$-trace terms}\right)=0\]
and this much is clear by construction.
\end{proof}

We are now in a position to consider the corresponding problem on
${\mathbb{CP}}_n$ as posed at the beginning of this section. Recall that
$Y^{\ell}$ denotes a certain tensor bundle bundle on any manifold but, in
particular, on~${\mathbb{CP}}_n$. Therefore, parallel to (\ref{thesecondrow})
on~${\mathbb{RP}}_n$, we may consider the sequence of bundles and linear
differential operators
\[\textstyle\bigodot^{\ell-1}\!\Lambda^1\xrightarrow{\,\nabla\,}
\bigodot^{\ell}\!\Lambda^1\xrightarrow{\,\nabla^{(\ell)}\,}Y^\ell\]
on~${\mathbb{CP}}_n$, where $\nabla^{(\ell)}$ is given by exactly the same
formula (\ref{RPcompatibility}) as on ${\mathbb{RP}}_n$ except that $\nabla_a$
now refers to the Fubini-Study connection and $g_{ab}$ to the Fubini-Study
metric. This sequence is no longer exact. Instead, if we expand the composition
$\nabla^{(\ell)}\circ\nabla$ using (\ref{curvatures}) on~${\mathbb{CP}}_n$,
bearing in mind that both $g_{ab}$ and $J_{ab}$ are covariant constant, and
compare the result with the total cancellation that we know occurs
on~${\mathbb{RP}}_n$, then we conclude that the result is forced to be of the
form 
\[\phi\mapsto J\circledcirc D\phi,\]
where $D:\bigodot^{\ell-1}\!\Lambda^1\to Y^{\ell-1}$ is some linear 
differential operator and
\[\circledcirc:\Lambda^2\otimes Y^{\ell-1}\to Y^\ell\]
is induced by projection onto the first factor in the decomposition
\[\begin{picture}(75,10)
\put(0,0){\line(1,0){35}}
\put(45,0){\makebox(0,0){$\cdots$}}
\put(55,0){\line(1,0){20}}
\put(0,0){\makebox(0,0){$\bullet$}}
\put(0,7){\makebox(0,0){\scriptsize$0$}}
\put(15,0){\makebox(0,0){$\bullet$}}
\put(15,7){\makebox(0,0){\scriptsize$1$}}
\put(30,0){\makebox(0,0){$\bullet$}}
\put(30,7){\makebox(0,0){\scriptsize$0$}}
\put(60,0){\makebox(0,0){$\bullet$}}
\put(60,7){\makebox(0,0){\scriptsize$0$}}
\put(75,0){\makebox(0,0){$\bullet$}}
\put(75,7){\makebox(0,0){\scriptsize$0$}}
\end{picture}\enskip\otimes\enskip\begin{picture}(85,10)
\put(0,0){\line(1,0){45}}
\put(55,0){\makebox(0,0){$\cdots$}}
\put(65,0){\line(1,0){20}}
\put(0,0){\makebox(0,0){$\bullet$}}
\put(0,7){\makebox(0,0){\scriptsize$0$}}
\put(20,0){\makebox(0,0){$\bullet$}}
\put(20,7){\makebox(0,0){\scriptsize$\ell-1$}}
\put(40,0){\makebox(0,0){$\bullet$}}
\put(40,7){\makebox(0,0){\scriptsize$0$}}
\put(70,0){\makebox(0,0){$\bullet$}}
\put(70,7){\makebox(0,0){\scriptsize$0$}}
\put(85,0){\makebox(0,0){$\bullet$}}
\put(85,7){\makebox(0,0){\scriptsize$0$}}
\end{picture}\enskip=\enskip\begin{picture}(75,10)
\put(0,0){\line(1,0){35}}
\put(45,0){\makebox(0,0){$\cdots$}}
\put(55,0){\line(1,0){20}}
\put(0,0){\makebox(0,0){$\bullet$}}
\put(0,7){\makebox(0,0){\scriptsize$0$}}
\put(15,0){\makebox(0,0){$\bullet$}}
\put(15,7){\makebox(0,0){\scriptsize$\ell$}}
\put(30,0){\makebox(0,0){$\bullet$}}
\put(30,7){\makebox(0,0){\scriptsize$0$}}
\put(60,0){\makebox(0,0){$\bullet$}}
\put(60,7){\makebox(0,0){\scriptsize$0$}}
\put(75,0){\makebox(0,0){$\bullet$}}
\put(75,7){\makebox(0,0){\scriptsize$0$}}
\end{picture}\enskip\oplus\cdots\]
of ${\mathrm{SL}}(2n,{\mathbb{R}})$-modules. In particular, we conclude that
the composition
\[\textstyle\bigodot^{\ell-1}\!\Lambda^1\xrightarrow{\,\nabla\,}
\bigodot^{\ell}\!\Lambda^1\xrightarrow{\,\nabla^{(\ell)}\,}Y^\ell
\xrightarrow{\,\perp\,}Y_\perp^\ell\] 
vanishes on~${\mathbb{CP}}_n$, where $\perp$ is the homomorphism of vector
bundles on ${\mathbb{CP}}_n$ induced by (\ref{totallytracefreepart}) and
$Y_\perp^\ell$ is induced by the irreducible
${\mathrm{Sp}}(2n,{\mathbb{R}})$-module
${}\;\begin{picture}(75,10)
\put(0,0){\line(1,0){35}}
\put(45,0){\makebox(0,0){$\cdots$}}
\put(55,0){\line(1,0){5}}
\put(60,1){\line(1,0){15}}
\put(60,-1){\line(1,0){15}}
\put(0,0){\makebox(0,0){$\bullet$}}
\put(0,7){\makebox(0,0){\scriptsize$0$}}
\put(15,0){\makebox(0,0){$\bullet$}}
\put(15,7){\makebox(0,0){\scriptsize$\ell$}}
\put(30,0){\makebox(0,0){$\bullet$}}
\put(30,7){\makebox(0,0){\scriptsize$0$}}
\put(60,0){\makebox(0,0){$\bullet$}}
\put(60,7){\makebox(0,0){\scriptsize$0$}}
\put(75,0){\makebox(0,0){$\bullet$}}
\put(75,7){\makebox(0,0){\scriptsize$0$}}
\put(67,0){\makebox(0,0){$\langle$}}
\end{picture}\;$. Writing $\nabla_\perp^{(\ell)}$ for the composition 
\begin{equation}\label{composition}
\textstyle\bigodot^{\ell}\!\Lambda^1\xrightarrow{\,\nabla^{(\ell)}\,}Y^\ell
\xrightarrow{\,\perp\,}Y_\perp^\ell,\end{equation}
we have proved the following.
\begin{theorem}\label{necessaryonCPn}
On ${\mathbb{CP}}_n$ for $n\geq 2$, there is a complex of 
linear differential operators
\[\textstyle\bigodot^{\ell-1}\!\Lambda^1\xrightarrow{\,\nabla\,}
\bigodot^{\ell}\!\Lambda^1\xrightarrow{\,\nabla_\perp^{(\ell)}\,}
Y_\perp^\ell.\] 
The operator $\nabla_\perp^{(\ell)}$ has the form
\[\pi_\perp\left(
\nabla_{(p}\nabla_q\nabla_r\cdots\nabla_{s)}\omega_{abc\cdots d}+
\mbox{\rm lower order $g$-trace terms}\right),\]
where $\pi_\perp$ is the composition 
$\bigodot^\ell\!\Lambda^1\otimes\bigodot^\ell\!\Lambda^1
\xrightarrow{\,\pi\,}Y^\ell\xrightarrow{\,\perp\,}Y_\perp^\ell$.
\end{theorem}
In particular, Theorem~\ref{necessaryonCPn} provides necessary conditions for a
globally defined smooth symmetric tensor $\omega_{ab\cdots c}$ on
${\mathbb{CP}}_n$ to be expressible in the form~$\nabla_{(a}\phi_{b\cdots c)}$
for some globally defined smooth symmetric tensor~$\phi_{b\cdots c}$. The
following section will show that these conditions are also sufficient.

\section{Sufficiency on~${\mathbb{CP}}_n$}

\subsection{The case $\ell=1$} In this case Theorem~\ref{necessaryonCPn} 
merely states that
\begin{equation}\label{beginRScomplex}
\Lambda^0\xrightarrow{\,d\,}\Lambda^1\xrightarrow{\,d_\perp\,}\Lambda_\perp^2
\end{equation}
is a complex on ${\mathbb{CP}}_n$, where $\Lambda_\perp^2$ denotes the bundle
of $2$-forms trace-free with respect to the K\"ahler form $J_{ab}$. This much
is clear and it is easy to identify the local cohomology of
(\ref{beginRScomplex}) as follows.
\begin{proposition} As a complex of sheaves, the cohomology of 
{\rm(\ref{beginRScomplex})} may be identified with the locally constant sheaf 
${\mathbb{R}}$. 
\end{proposition}
\begin{proof}
Suppose $\omega$ is a locally defined $1$-form with $d_\perp\omega=0$. This 
means that $d\omega=\theta J$ for some smooth function~$\theta$. Applying $d$ 
gives
\[0=d\theta\wedge J+\theta\, dJ=d\theta\wedge J\]
because $J$ is closed. But since $J$ is non-degenerate and $n\geq 2$, it 
follows by linear algebra that $d\theta=0$. Hence $\theta$ is locally 
constant. As $J$ is closed, locally we may choose a $1$-form $\alpha$ so that 
$J=d\alpha$. Then $d(\omega-\theta\alpha)=0$ and we conclude that locally we 
may always write
\begin{equation}\label{decomp}
\omega=d\phi+\theta\alpha\qquad\mbox{for some some function $\phi$},
\end{equation}
where $d\alpha=J$ and $\theta$ is locally constant. Although $\alpha$ is not
determined by $J$, the only freedom in its choice is
$\alpha\mapsto\alpha+d\psi$ for some smooth function~$\psi$, which may be 
absorbed into the decomposition (\ref{decomp}) as
\[\omega=d\phi+\theta\alpha=d(\phi-\theta\psi)+\theta(\alpha+d\psi).\]
In particular, the locally constant function $\theta$ is well defined by the 
local cohomology of the complex~(\ref{beginRScomplex}).
\end{proof}
Globally on ${\mathbb{CP}}_n$, however, there is no $1$-form $\alpha$ with
$d\alpha=J$. Hence, as was already observed by J.-P.~Demailly
(cf.~\cite[Introduction]{gg3} or \cite[Theorem~3.40]{gg6}), the same line of
argument shows that globally there is no cohomology. In other words, we have
proved our desired global result as follows.
\begin{theorem}\label{desired}
The following complex
\[\Gamma({\mathbb{CP}}_n,\Lambda^0)\xrightarrow{\,d\,}
\Gamma({\mathbb{CP}}_n,\Lambda^1)\xrightarrow{\,d_\perp\,}
\Gamma({\mathbb{CP}}_n,\Lambda_\perp^2)\]
is exact.\end{theorem}
For the cases $\ell\geq 2$, it will be useful to extend (\ref{beginRScomplex}) 
as follows. Let us firstly suppose that $n\geq 3$. Then there is a naturally 
defined complex
\begin{equation}\label{nextRScomplex}
\Lambda^0\xrightarrow{\,d\,}\Lambda^1\xrightarrow{\,d_\perp\,}\Lambda_\perp^2
\xrightarrow{\,d_\perp\,}\Lambda_\perp^3
\end{equation}
where $\Lambda_\perp^3$ denotes the bundle of $3$-forms trace-free with 
respect to $J_{ab}$ and $d_\perp:\Lambda_\perp^2\to\Lambda_\perp^3$ is defined 
as the composition
\[\Lambda_\perp^2\hookrightarrow\Lambda^2\xrightarrow{\,d\,}\Lambda^3=
\Lambda_\perp^3\oplus\Lambda^1\to\Lambda_\perp^3.\]
In case $n=2$ notice that $\Lambda_\perp^3=0$. Otherwise, we have found the 
integrability conditions for the range of 
$d_\perp:\Lambda^1\to\Lambda_\perp^2$. 
\begin{proposition}\label{exact}
On ${\mathbb{CP}}_n$ for $n\geq 3$, the complex
\[\Lambda^1\xrightarrow{\,d_\perp\,}\Lambda_\perp^2
\xrightarrow{\,d_\perp\,}\Lambda_\perp^3\]
is exact on the level of sheaves.\end{proposition}
\begin{proof}
Suppose $\xi$ is a locally defined $J$-trace-free $2$-form with $d_\perp\xi=0$.
This means that $d\xi=\mu\wedge J$ for some $1$-form~$\mu$. Applying $d$ gives
\[0=d\mu\wedge J-\mu\wedge dJ=d\mu\wedge J\]
because $J$ is closed. But since $J$ is non-degenerate and $n\geq 3$, it
follows by linear algebra that $d\mu=0$. Locally, therefore, we may find a 
smooth function $\phi$ with $d\phi=\mu$. Hence $d(\xi-\phi J)=0$ and locally 
we may find a smooth $1$-form $\omega$ such that $\xi-\phi J=d\omega$. Since 
$\xi$ is $J$-trace-free, we conclude that $\xi=d_\perp\omega$.
\end{proof}
When $n=2$ there is a replacement for (\ref{nextRScomplex}) due to M.~Rumin and
N.~Seshadri~\cite{rs} and defined (on any $4$-dimensional symplectic manifold)
as follows. Suppose $\xi$ is a smooth $2$-form and consider $d\xi$.
Since~$n=2$, there is an isomorphism
\[\Lambda^1\xrightarrow{\,\underbar{\enskip}\wedge J\,}\Lambda^3\]
so we may write $d\xi=\mu\wedge J$ for a uniquely defined $1$-form~$\mu$. 
Applying $d$ implies $d\mu\wedge J=0$ whence $d\mu$ is $J$-trace-free. Let us 
write $d_\perp^{(2)}:\Lambda^2\to\Lambda_\perp^2$ for the resulting 
differential operator. Specifically, 
\begin{equation}\label{defofdperp2}
d_\perp^{(2)}\xi=d\mu,\quad\mbox{where }\mu\wedge J=d\xi.\end{equation}
Notice that if $\xi=\theta J$, then $\mu=d\theta$ and so $d_\perp^{(2)}\xi=0$. 
Thus, we obtain a complex of differential operators on~${\mathbb{CP}}_2$
\[\Lambda^0\xrightarrow{\,d\,}
\Lambda^1\xrightarrow{\,d_\perp\,}\Lambda_\perp^2
\xrightarrow{\,d_\perp^{(2)}\,}\Lambda_\perp^2,\]
which acts as a replacement for (\ref{nextRScomplex}), especially in view of 
the following replacement for Proposition~\ref{exact}.
\begin{proposition}
On ${\mathbb{CP}}_2$, the complex
\[\Lambda^1\xrightarrow{\,d_\perp\,}\Lambda_\perp^2
\xrightarrow{\,d_\perp^{(2)}\,}\Lambda_\perp^2\]
is exact on the level of sheaves.\end{proposition}
\begin{proof} If $d_\perp^{(2)}\xi=0$ in (\ref{defofdperp2}), then locally we 
may write $d\xi=d\phi\wedge J$ for some smooth function~$\phi$. In this case
$d(\xi-\phi J)=0$ and locally we may find a smooth $1$-form $\omega$ such 
that $\xi=d\omega+\phi J$. If $\xi$ is also $J$-trace-free, then it is 
immediate that $\xi=d_\perp\omega$.
\end{proof}

\subsection{The case $\ell=2$}\label{ellequalstwo}
Let ${\mathbb{U}}$ denote the bundle $\Lambda^0\oplus\Lambda^1\oplus\Lambda^0$ 
on ${\mathbb{CP}}_n$ equipped with the connection
\begin{equation}\label{newtractors}
{\mathbb{U}}=\begin{array}c\Lambda^0\\ \oplus\\ \Lambda^1\\
\oplus\\
 \Lambda^0\end{array}\ni
\left[\begin{array}c\sigma\\ \mu_b\\ \rho\end{array}\right]
\stackrel{\nabla_a}{\longmapsto}
\left[\begin{array}c\nabla_a\sigma-\mu_a\\ 
\nabla_a\mu_b+g_{ab}\sigma+J_{ab}\rho\\
\nabla_a\rho-J_a{}^c\mu_c
\end{array}\right]\in\Lambda^1\otimes{\mathbb{U}}.\end{equation}
We compute that
\[\nabla_a\nabla_b
\left[\begin{array}c\sigma\\ \mu_c\\ \rho\end{array}\right]=
\left[\begin{array}c
J_{ab}\rho+\ldots\\ 
\nabla_a\nabla_b\mu_c-g_{ac}\mu_b-J_{ac}J_b{}^d\mu_d+\ldots\\
-J_{ab}\sigma+\ldots
\end{array}\right],\]
where the ellipses $\ldots$ denote tensors symmetric in the $ab$ indices. 
{From} (\ref{curvatures}) we see that
\[(\nabla_a\nabla_b-\nabla_b\nabla_a)\mu_c=
g_{ac}\mu_b-g_{bc}\mu_a+J_{ac}J_b{}^d\mu_d-J_{bc}J_a{}^d\mu_d
+2J_{ab}J_c{}^d\mu_d\]
and hence that the curvature of the connection on ${\mathbb{U}}$ is given by
\begin{equation}\label{curvatureofU}(\nabla_a\nabla_b-\nabla_b\nabla_a)
\left[\begin{array}c\sigma\\ \mu_c\\ \rho\end{array}\right]=
2J_{ab}\left[\begin{array}c\rho\\ J_c{}^d\mu_d\\ -\sigma\end{array}\right].
\end{equation}
In other words, the curvature of this connection on ${\mathbb{U}}$ has the form
\begin{equation}\label{symplecticEinstein}
(\nabla_a\nabla_b-\nabla_b\nabla_a)\Sigma=2J_{ab}\Phi\Sigma,\end{equation}
where $\Phi$ is some endomorphism of~${\mathbb{U}}$. 

It is easily verified that the skew form on ${\mathbb{U}}$ defined by 
\begin{equation}\label{symplecticreduction}
\langle(\sigma,\mu_a,\rho),(\tilde\sigma,\tilde\mu_b,\tilde\rho)\rangle
=\sigma\tilde\rho+J^{ab}\mu_a\tilde\mu_b-\rho\tilde\sigma\end{equation}
is compatible with $\nabla_a$ in the sense that  
$\nabla_a\langle\Sigma,\tilde\Sigma\rangle=
\langle\nabla_a\Sigma,\tilde\Sigma\rangle+
\langle\Sigma,\nabla_a\tilde\Sigma\rangle$. 
Hence, the structure group for ${\mathbb{U}}$ can be reduced to 
${\mathrm{Sp}}(2(n+1),{\mathbb{R}})$ and, just as we did for the connection
on ${\mathbb{T}}$ in~\S\ref{necessary}, we may now consider the connections on
vector bundles induced from~${\mathbb{U}}$ by the irreducible representations 
of this structure group. Parallel to $\Lambda^2{\mathbb{T}}$ in 
\S\ref{necessary}, we should consider $\Lambda_\perp^2{\mathbb{U}}$ where 
$\perp$ denotes the trace-free part with respect
to~(\ref{symplecticreduction}). Its connection is easily computed
\begin{equation}\label{skewtracefreetractors}
\begin{array}l\Lambda_\perp^2{\mathbb{U}}=
\!\begin{array}c\Lambda^1\\ \oplus\\ \Lambda^2\\ \oplus\\ \Lambda^1\end{array}
\!\ni\left[\begin{array}c\sigma_b\\ \mu_{bc}\\ \rho_b\end{array}\right]
\begin{picture}(0,0)\put(0,0){\line(1,0){20}}
\put(0,-2){\line(0,1){4}}
\put(20,0){\vector(0,-1){20}}
\put(10,6){\makebox(0,0){\small$\nabla_a$}}\end{picture}\\
\left[\begin{array}c\nabla_a\sigma_b-\mu_{ab}\\ 
\nabla_a\mu_{bc}+g_{ab}\sigma_c-g_{ac}\sigma_b+J_{ab}\rho_c
-J_{ac}\rho_b-J_{bc}\rho_a+J_{bc}J_a{}^d\sigma_d\\
\nabla_a\rho_b+J_a{}^d\mu_{bd}
\end{array}\right]\end{array}\end{equation}
and we immediately notice the similarity with~(\ref{skewtractors}). The 
curvature of this connection automatically has the form 
(\ref{symplecticEinstein}) with $\Phi$ replaced by the induced endomorphism 
of~$\Lambda_\perp^2{\mathbb{U}}$. Alternatively, it may be verified by 
composition with the induced operator
$\nabla:\Lambda^1\otimes\Lambda_\perp^2{\mathbb{U}}\to
\Lambda^2\otimes\Lambda_\perp^2{\mathbb{U}}$ given by
\begin{equation}\label{coupledskewtracefreetractors}
\raisebox{-20pt}{\makebox[0pt]{$\begin{array}l
\hspace*{80pt}
\left[\begin{array}c\sigma_{bc}\\ \mu_{bcd}\\ \rho_{bc}\end{array}\right]
\begin{picture}(0,0)\put(0,0){\line(1,0){20}}
\put(0,-2){\line(0,1){4}}
\put(20,0){\vector(0,-1){20}}
\end{picture}\\[20pt]
\mbox{\small{}\!$\left[\!\!\begin{array}c
\nabla_{[a}\sigma_{b]c}+\mu_{[ab]c}\\ 
\nabla_{[a}\mu_{b]cd}+g_{c[a}\sigma_{b]d}-g_{d[a}\sigma_{b]c}
-J_{c[a}\rho_{b]d}+J_{d[a}\rho_{b]c}+J_{cd}\rho_{[ab]}
+J_{cd}J_{[a}{}^e\sigma_{b]e}\\ 
\nabla_{[a}\rho_{b]c}+J_{[a}{}^e\mu_{b]ce}
\end{array}\!\!\right]$\!}\end{array}$}}\end{equation}
that
\begin{equation}\label{curvatureskewtracefreetractors}
(\nabla_a\nabla_b-\nabla_b\nabla_a)
\left[\begin{array}c\sigma_c\\ \mu_{cd}\\ \rho_c\end{array}\right]=
2J_{ab}\left[\begin{array}c\rho_c+J_c{}^e\sigma_e\\ 
J_c{}^e\mu_{ed}+J_d{}^e\mu_{ce}\\ -\sigma_c+J_c{}^e\rho_e\end{array}\right].
\end{equation}

\begin{theorem}\label{splittingCPn}
Suppose $\omega_{ab}$ is a symmetric tensor on~${\mathbb{CP}}_n$. The following 
are (locally or globally) equivalent.
\begin{enumerate}\setlength{\itemsep}{3pt}
\item $\left[\begin{array}c\omega_{bc}\\ 
\nabla_c\omega_{db}-\nabla_d\omega_{cb}\\
{\mathcal{L}}_{bc}(\omega)
\end{array}\right]
\in\Gamma({\mathbb{CP}}_n,\Lambda^1\otimes\Lambda_\perp^2{\mathbb{U}})$
is in the range of the induced  connection 
$\nabla_b:\Gamma({\mathbb{CP}}_n,\Lambda_\perp^2{\mathbb{U}})\longrightarrow
\Gamma({\mathbb{CP}}_n,\Lambda^1\otimes\Lambda_\perp^2{\mathbb{U}})$, where 
$\bigodot^2\!\Lambda^1\in\omega_{bc}\mapsto{\mathcal{L}}_{bc}(\omega)\in
\Lambda^1\otimes\Lambda^1$ is some explicit linear differential operator (to be 
determined in the proof).
\item $\omega_{ab}=\nabla_{(a}\phi_{b)}$ for some 
$\phi_a\in\Gamma({\mathbb{CP}}_n,\Lambda^1)$.
\item $\left[\begin{array}c\omega_{bc}\\ 
\mu_{bcd}\\ \rho_{bc}\end{array}\right]
\in\Gamma({\mathbb{CP}}_n,\Lambda^1\otimes\Lambda_\perp^2{\mathbb{U}})$,
for some $\mu_{bcd}\in\Gamma({\mathbb{CP}}_n,\Lambda^1\otimes\Lambda^2)$ and
$\rho_{bc}\in\Gamma({\mathbb{CP}}_n,\Lambda^1\otimes\Lambda^1)$,
is in the range of the connection 
$\nabla_b:\Gamma({\mathbb{CP}}_n,\Lambda_\perp^2{\mathbb{U}})\to
\Gamma({\mathbb{CP}}_n,\Lambda^1\otimes\Lambda_\perp^2{\mathbb{U}})$.
\end{enumerate}
\end{theorem}
\begin{proof}
It is clear from (\ref{skewtracefreetractors}) that
(i)$\Rightarrow$(ii)$\Rightarrow$(iii). It remains to show 
(iii)$\Rightarrow$(i). To see this, recall that the curvature of the 
connection on $\Lambda_\perp^2{\mathbb{U}}$ has the form 
(\ref{symplecticEinstein}) and so if (iii) holds, then we read off from the 
first row of (\ref{coupledskewtracefreetractors}) that 
\[\nabla_{[a}\omega_{b]c}+\mu_{[ab]c}=J_{ab}\rho_c\]
for some~$\rho_c$.
{From} this, bearing in mind that $\mu_{bcd}=\mu_{[bcd]}$, it follows that 
\[\mu_{bcd}=3\mu_{[bcd]}-2\mu_{[cd]b}=
\nabla_{c}\omega_{db}-\nabla_{d}\omega_{cb}
+J_{bc}\rho_d-J_{cd}\rho_b-J_{bd}\rho_c\]
and from (\ref{skewtracefreetractors}) we see that 
\[\left[\begin{array}c\omega_{bc}\\ 
\mu_{bcd}\\ \rho_{bc}\end{array}\right]=\left[\begin{array}c\omega_{bc}\\ 
\nabla_{c}\omega_{db}-\nabla_{d}\omega_{cb}\\ 
\ast\end{array}\right]
+\nabla_b\left[\begin{array}c 0\\ 0\\ \rho_{c}\end{array}\right]\]
for some $\rho_c$. As the second term on the right hand side is already of the 
required form, it follows that we may take 
$\mu_{bcd}=\nabla_{c}\omega_{db}-\nabla_{d}\omega_{cb}$ without loss of 
generality and it remains to consider~$\rho_{bc}$. In fact, we claim that 
$\rho_{bc}$ now is uniquely determined by equating the 
second row of (\ref{coupledskewtracefreetractors}) to $J_{ab}\tau_{cd}$ for 
some $\tau_{cd}=\tau_{[cd]}$, as must be the case 
by~(\ref{symplecticEinstein}). To see this, we need the following purely 
algebraic result.
\begin{lemma}\label{purealgebra}
Suppose $T_{abcd}$ is a tensor with the following symmetries
\begin{itemize} 
\item $T_{abcd}=T_{[ab][cd]}$,
\item $T_{[abc]d}=J_{[ab}\psi_{c]d}$ for some tensor $\psi_{cd}$.
\end{itemize}
Then there are unique tensors 
\begin{itemize}
\item $\rho_{ab}$,
\item $\tau_{ab}=\tau_{[ab]}$,
\item $X_{abcd}=X_{[ab][cd]}$ with $X_{[abc]d}=0$ and $J^{ab}X_{abcd}=0$,
\end{itemize}
such that 
\begin{equation}\label{whatwewant}
T_{abcd}=X_{abcd}+J_{c[a}\rho_{b]d}-J_{d[a}\rho_{b]c}-J_{cd}\rho_{[ab]}
+J_{ab}\tau_{cd}.\end{equation}
\end{lemma}
\begin{proof} Let us consider a complex consisting of (quotients of) spaces of 
tensors for ${\mathrm{Sp}}(2n,{\mathbb{R}})$ and homomorphisms between them
\begin{equation}\label{trickycomplex}0\to A\to B\to C\to 0\end{equation}
defined by setting
\begin{itemize}
\item $A=\{\rho_{bc}\mbox{ with no particular symmetries}\}$
\item $B=\{T_{abcd}\mbox{ s.t.\ }T_{abcd}=T_{[ab][cd]}\}/
\{T_{abcd}=J_{ab}\tau_{cd}\mbox{ for }\tau_{cd}=\tau_{[cd]}\}$
\item $C=\{S_{abcd}\mbox{ s.t.\ }S_{abcd}=S_{[abc]d}\}/
\{S_{abcd}=J_{[ab}\psi_{c]d}\mbox{ for some }\psi_{cd}\}$
\end{itemize}
and taking 
\[A\ni\rho_{bc}\mapsto 
J_{c[a}\rho_{b]d}-J_{d[a}\rho_{b]c}-J_{cd}\rho_{[ab]}\in B\ni 
T_{abcd}\mapsto T_{[abc]d}\in C.\]
It is easily verified that 
\begin{itemize}
\item this is, indeed, a complex,
\item $A\to B$ is injective,
\item $B\to C$ is surjective,
\item if we set 
$H=\{[X_{abcd}]\in B\mbox{ s.t.\ }X_{[abc]d}=0\mbox{ and }J^{ab}X_{abcd}=0\}$,
then $H$ is transverse to the image of $A\hookrightarrow B$ and is mapped to 
zero under $B\to C$.
\end{itemize}
Lemma~\ref{purealgebra} is precisely the statement that $H$ represents the
cohomology of the complex~(\ref{trickycomplex}). This is most easily verified 
by computing dimensions
\begin{itemize}
\item $\dim A
=4n^2$,
\item $\dim B
=(n-1)n(2n-1)(2n+1)$,
\item $\displaystyle\dim C
=4(n-2)n^2(2n+1)/3,$
\item $\displaystyle\dim H 
=(n-1)n(2n-1)(2n+3)/3,$
\end{itemize}
(for $n\geq 2$) and the result follows.
\end{proof}
\noindent To continue the proof of Theorem~\ref{splittingCPn}, we claim that
Lemma~\ref{purealgebra} applies to
\begin{equation}\label{thisisT}T_{abcd}=
\nabla_{[a}\mu_{b]cd}+g_{c[a}\omega_{b]d}-g_{d[a}\omega_{b]c}
+J_{cd}J_{[a}{}^e\omega_{b]e},\end{equation}
where $\mu_{bcd}=\nabla_c\omega_{bd}-\nabla_d\omega_{bc}$. This is because
\[T_{[abc]d}=2J_{[ab}\omega_{c]e}J_d{}^e,\]
as can be verified by direct computation from (\ref{curvatures}) or, more 
simply, by noticing that 
\[\Lambda^1\otimes\Lambda_\perp^2{\mathbb{U}}\ni
\left[\!\begin{array}c\omega_{bc}\\ 
\nabla_{c}\omega_{db}-\nabla_{d}\omega_{cb}\\ 
0\end{array}\!\right]\stackrel{\nabla}{\mapsto}
\left[\!\begin{array}c0\\ 
T_{abcd}\\ 
\ast\end{array}\!\right]\stackrel{\nabla}{\mapsto}
\left[\!\begin{array}cT_{[abc]d}\\ 
\ast\\ 
\ast\end{array}\!\right]\in\Lambda^3\otimes\Lambda_\perp^2{\mathbb{U}}\]
and that the curvature of $\nabla$ on $\Lambda_\perp^2{\mathbb{U}}$ is given
by~(\ref{curvatureskewtracefreetractors}) (but, in fact, we only need to know
that the curvature has the form (\ref{symplecticEinstein}) in order to see that
$T_{[abc]d}=J_{[ab}\psi_{c]d}$ for some $\psi_{cd}$ and be in a position to
apply Lemma~\ref{purealgebra}). We conclude from Lemma~\ref{purealgebra} that
$T_{abcd}$ in (\ref{thisisT}) uniquely determines tensor fields $\rho_{ab}$,
$\tau_{ab}$, and $X_{abcd}$ satisfying the symmetries specified in
Lemma~\ref{purealgebra} and such that (\ref{whatwewant}) holds. Of course, we 
could determine $\rho_{ab}$ explicitly from its characterising properties and 
especially (\ref{whatwewant}) by tracing over various pairs of indices 
using the inverse symplectic form~$J^{ab}$. The result is extraordinarily 
complicated but has the form
\[\rho_{ab}=
-\frac{1}{2(n+1)}\left(S_{ab}-\frac{1}{2n+1}J^{cd}S_{cd}J_{ab}\right)
+\mbox{lower order terms},\]
where 
$S_{ab}=J^{cd}(\nabla_a\nabla_c\omega_{bd}-\nabla_b\nabla_c\omega_{ad})$. The 
mapping $\omega_{bc}\mapsto\rho_{bc}$ defines the differential operator
${\mathcal{L}}_{bc}$ in the statement of Theorem~\ref{splittingCPn} and from 
(\ref{coupledskewtracefreetractors}) we have arranged that
\begin{equation}\label{whatwevearranged}
\raisebox{-16pt}{\makebox[0pt]{$\Lambda^1\otimes\Lambda_\perp^2{\mathbb{U}}\ni
\left[\begin{array}c\omega_{bc}\\ 
\nabla_{c}\omega_{db}-\nabla_{d}\omega_{cb}\\ 
{\mathcal{L}}_{bc}(\omega)\end{array}\right]\stackrel{\nabla}{\longmapsto}
\left[\begin{array}c0\\ 
X_{abcd}+J_{ab}\tau_{cd}\\ 
\ast\end{array}\right]\in\Lambda^2\otimes\Lambda_\perp^2{\mathbb{U}},$}}
\end{equation}
where $X_{abcd}$ and $\tau_{cd}$ are determined by~(\ref{whatwewant}).
Recall that this conclusion was derived under assumption (iii) in 
Theorem~\ref{splittingCPn} with the additional constraint, without losing
generality, that $\mu_{bcd}=\nabla_{c}\omega_{db}-\nabla_{d}\omega_{cb}$. The 
conclusion that~$\rho_{bc}={\mathcal{L}}_{bc}(\omega)$, for the differential 
operator ${\mathcal{L}}_{bc}$ derived above, was forced by 
these assumptions. This is enough to complete the proof of 
Theorem~\ref{splittingCPn}.
\end{proof}
There are, however, some further conclusions that can be derived from this 
proof and are worth recording here. Firstly, if (iii) holds, then it is 
immediate from (\ref{symplecticEinstein}) and (\ref{whatwevearranged}), that 
$X_{abcd}=0$. We can compute $X_{abcd}$ from (\ref{whatwewant}) and 
(\ref{thisisT}) but, in fact, we have essentially done this computation already 
in \S\ref{necessary} when we derived (\ref{rawobstruction}) leading 
to~(\ref{refinedobstruction}). The point is that if one simply ignores all 
terms involving $J$ in the formula (\ref{skewtracefreetractors}) for the
connection on $\Lambda_\perp^2{\mathbb{U}}$, then one obtains
\[\left[\begin{array}c\sigma_b\\ \mu_{bc}\\ \rho_b\end{array}\right]
\longmapsto
\left[\begin{array}c\nabla_a\sigma_b-\mu_{ab}\\ 
\nabla_a\mu_{bc}+g_{ab}\sigma_c-g_{ac}\sigma_b\\
\nabla_a\rho_b
\end{array}\right],\]
which, as far as the first two rows are concerned, coincides 
with~(\ref{skewtractors}). But we are planning to remove all $J$-traces in 
defining $X_{abcd}$ via~(\ref{whatwewant}). Bearing in mind that the Riemann 
curvature tensors on complex and real projective space differ only by terms 
involving~$J$, it follows from the derivation of (\ref{refinedobstruction}) 
that 
\begin{equation}\label{X}X_{abcd}=
2\times\pi_\perp(\nabla_{(a}\nabla_{c)}\omega_{bd}+g_{ac}\omega_{bd}),
\end{equation}
where recall that the subscript $\perp$ means to remove the $J$-traces. Of 
course, this confirms Theorem~\ref{necessaryonCPn} in case~$\ell=2$. Another 
key observation from (\ref{whatwevearranged}) is as follows.
\begin{theorem}\label{observation}
Suppose $\omega_{ab}$ is a symmetric tensor on ${\mathbb{CP}}_n$ and that
\begin{equation}\label{Xvanishes}
\pi_\perp(\nabla_{(a}\nabla_{c)}\omega_{bd}+g_{ac}\omega_{bd})=0.
\end{equation}
Then
\begin{equation}\label{forced}\left[\begin{array}c\omega_{bc}\\ 
\nabla_{c}\omega_{db}-\nabla_{d}\omega_{cb}\\ 
{\mathcal{L}}_{bc}(\omega)\end{array}\right]\stackrel{\nabla}{\longmapsto}
\left[\begin{array}c0\\ 
J_{ab}\tau_{cd}\\ 
J_{ab}\theta_c
\end{array}\right]
\in\Gamma({\mathbb{CP}}_n,\Lambda^2\otimes\Lambda_\perp^2{\mathbb{U}}),
\end{equation}
for some $\tau_{cd}\in\Gamma({\mathbb{CP}}_n,\Lambda^2)$ and 
$\theta_c\in\Gamma({\mathbb{CP}}_n,\Lambda^1)$.
\end{theorem}
\begin{proof} 
The first and second rows of the left hand side of (\ref{forced}) are 
enough to give the vanishing of the first row of the right hand side and 
${\mathcal{L}}_{bc}(\omega)$ is designed so that (\ref{whatwevearranged}) 
holds in this case. Now that we know~(\ref{Xvanishes}), it follows from  
(\ref{X}) that the second row of the right hand side of (\ref{forced}) is of 
the stated form. It remains to demonstrate that the third row is as stated. 
Evidently, given the difficulties in writing down an explicit formula for 
${\mathcal{L}}_{bc}(\omega)$, a direct verification is out of the question. 
Instead, let us observe that 
\[\left[\begin{array}c0\\ 
J_{bc}\tau_{de}\\ 
\rho_{bcd}
\end{array}\right]
\stackrel{\nabla}{\longmapsto}
\left[\begin{array}c
-J_{[ab}\tau_{c]d}\\ 
J_{[ab}\nabla_{c]}\tau_{de}
-J_{d[a}\rho_{bc]e}+J_{e[a}\rho_{bc]d}-J_{de}\rho_{[abc]}\\ 
\nabla_{[a}\rho_{bc]d}+J_{[ab}J_{c]}{}^f\tau_{df}
\end{array}\right]\]
under $\Lambda^2\otimes\Lambda_\perp^2{\mathbb{U}}
\xrightarrow{\,\nabla\,}\Lambda^3\otimes\Lambda_\perp^2{\mathbb{U}}$ and from 
(\ref{symplecticEinstein}) conclude that
\begin{equation}\label{know}
-J_{d[a}\rho_{bc]e}+J_{e[a}\rho_{bc]d}-J_{de}\rho_{[abc]}
=J_{[ab}\psi_{c]de}\end{equation}
for some $\psi_{cde}=\psi_{c[de]}$. Certainly (\ref{know}) holds if 
$\rho_{cde}=J_{cd}\theta_e$. Conversely, it may be verified without too much 
difficulty that the 
converse is true on~${\mathbb{CP}}_n$ for $n\geq 3$. On ${\mathbb{CP}}_2$ 
(\ref{know}) is content-free and a separate argument is needed: it turns out 
that there is a second order differential operator on 
$\Lambda^2\otimes\Lambda_\perp^2{\mathbb{U}}$ that may be applied to give 
sufficiently strong algebraic constraints on~$\rho_{bcd}$. This is part of the 
general theory developed in \S\ref{generalell} below and will be
omitted here (and the general theory will also provide a workaround for the
algebraic verification claimed above).
\end{proof}

We are at last in a position to prove the sufficiency of the condition given in 
Theorem~\ref{necessaryonCPn} in case $\ell=2$.

\begin{theorem}\label{prototypesufficiencyonCPn}
Suppose $\omega_{ab}$ is a globally defined smooth symmetric tensor
on ${\mathbb{CP}}_n$ and that
\[\pi_\perp(\nabla_{(a}\nabla_{c)}\omega_{bd}+g_{ac}\omega_{bd})=0.\]
Then there is a smooth $1$-form $\phi_a$ on ${\mathbb{CP}}_n$ such that 
$\omega_{ab}=\nabla_{(a}\phi_{b)}$.
\end{theorem}
\begin{proof}
According to Theorems~\ref{splittingCPn} and~\ref{observation} it suffices to 
show that if
\[\Gamma({\mathbb{CP}}_n,\Lambda^1\otimes\Lambda_\perp^2{\mathbb{U}})\ni
\left[\begin{array}c\omega_{bc}\\ \mu_{bcd}\\ \rho_{bc}\end{array}\right]
\stackrel{\nabla_a}{\longmapsto}
\left[\begin{array}c0\\ 
J_{ab}\tau_{cd}\\ 
J_{ab}\theta_c
\end{array}\right]\in
\Gamma({\mathbb{CP}}_n,\Lambda^2\otimes\Lambda_\perp^2{\mathbb{U}})\]
for some $\tau_{cd}\in\Gamma({\mathbb{CP}}_n,\Lambda^2)$ and 
$\theta_c\in\Gamma({\mathbb{CP}}_n,\Lambda^1)$, then 
\[\left[\begin{array}c\omega_{bc}\\ \mu_{bcd}\\ \rho_{bc}\end{array}\right]
\mbox{ is in the range of }
\nabla_b:\Gamma({\mathbb{CP}}_n,\Lambda_\perp^2{\mathbb{U}})\to
\Gamma({\mathbb{CP}}_n,\Lambda^1\otimes\Lambda_\perp^2{\mathbb{U}}).\]
In other words, it suffices to show exactness of the complex
\[\Gamma({\mathbb{CP}}_n,\Lambda_\perp^2{\mathbb{U}})\xrightarrow{\,\nabla\,}
\Gamma({\mathbb{CP}}_n,\Lambda^1\otimes\Lambda_\perp^2{\mathbb{U}})
\xrightarrow{\,\nabla_\perp\,}
\Gamma({\mathbb{CP}}_n,\Lambda_\perp^2\otimes\Lambda_\perp^2{\mathbb{U}}).\]
We already used (\ref{curvatureskewtracefreetractors}) in observing 
that this is, indeed, a complex but now let us analyse 
(\ref{curvatureskewtracefreetractors}) more precisely. Certainly, it is of the 
form (\ref{symplecticEinstein}) for some 
$\Phi:\Lambda_\perp^2{\mathbb{U}}\to\Lambda_\perp^2{\mathbb{U}}$ but this 
$\Phi$ is quite special. {From} (\ref{curvatureskewtracefreetractors}) we 
compute that 
\[\left[\begin{array}c\sigma_c\\ \mu_{cd}\\ \rho_c\end{array}\right]
\stackrel{\Phi^2}{\longmapsto}
2\left[\begin{array}c-\sigma_c+J_c{}^e\rho_e\\ 
-\mu_{cd}+J_c{}^eJ_d{}^f\mu_{ef}\\ 
-\rho_c-J_c{}^e\sigma_e\end{array}\right],\]
which suggests the decomposition
\begin{equation}\label{decomposition}
\Lambda_\perp^2{\mathbb{U}}=
\Lambda_{\perp,0}^2{\mathbb{U}}\oplus\Lambda_{\perp,-4}^2{\mathbb{U}}
\end{equation}
according to
\[\left[\begin{array}c\sigma_c\\ \mu_{cd}\\ \rho_c\end{array}\right]=
\frac12\left[\begin{array}c
\sigma_c+J_c{}^e\rho_e\\ \mu_{cd}+J_c{}^eJ_d{}^f\mu_{ef}
\\ \rho_c-J_c{}^e\sigma_e\end{array}\right]
+\frac12\left[\begin{array}c
\sigma_c-J_c{}^e\rho_e\\ \mu_{cd}-J_c{}^eJ_d{}^f\mu_{ef}
\\ \rho_c+J_c{}^e\sigma_e\end{array}\right]\]
for then $\Phi^2$ vanishes on $\Lambda_{\perp,0}^2{\mathbb{U}}$ and coincides
with $-4\times{\mathrm{Id}}$ on~$\Lambda_{\perp,-4}^2{\mathbb{U}}$.
Furthermore, it is readily verified that the connection $\nabla$ respects this
decomposition. We are reduced to showing exactness of the following two
complexes:--
\begin{equation}\label{zero}
\Gamma({\mathbb{CP}}_n,\Lambda_{\perp,0}^2{\mathbb{U}})\stackrel{\nabla}{\to}
\Gamma({\mathbb{CP}}_n,\Lambda^1\otimes\Lambda_{\perp,0}^2{\mathbb{U}})
\stackrel{\nabla\!\!{}_\perp}{\to}
\Gamma({\mathbb{CP}}_n,\Lambda_\perp^2\otimes\Lambda_{\perp,0}^2{\mathbb{U}})
\end{equation}
and
\begin{equation}\label{minusfour}
\Gamma({\mathbb{CP}}_n,\Lambda_{\perp,-4}^2{\mathbb{U}})
\!\stackrel{\nabla}{\to}\!
\Gamma({\mathbb{CP}}_n,\Lambda^1\otimes\Lambda_{\perp,-4}^2{\mathbb{U}})
\!\stackrel{\nabla\!\!{}_\perp}{\to}\!
\Gamma({\mathbb{CP}}_n,\Lambda_\perp^2\otimes\Lambda_{\perp,-4}^2{\mathbb{U}}).
\end{equation}
The exactness of (\ref{minusfour}) is straightforward as follows. Suppose 
\[\Omega_a\in
\Gamma({\mathbb{CP}}_n,\Lambda^1\otimes\Lambda_{\perp,-4}^2{\mathbb{U}})\mbox{ 
satisfies }\nabla_\perp\Omega=0.\]
Then $\nabla_{[a}\Omega_{b]}=J_{ab}\Sigma$ for some
$\Sigma\in\Gamma({\mathbb{CP}}_n,\Lambda_{\perp,-4}^2{\mathbb{U}})$. {From} 
(\ref{symplecticEinstein}) it follows that 
$\nabla_{[a}\nabla_{b]}\Omega_c=J_{ab}\Phi\Omega_c$ whence
\[J_{[ab}\Phi\Omega_{c]}=\nabla_{[a}\nabla_{b}\Omega_{c]}=
\nabla_{[a}(J_{bc]}\Sigma)=J_{[ab}\nabla_{c]}\Sigma\]
and, since $\Lambda^1\xrightarrow{\,J\wedge\underbar{\enskip}\,}\Lambda^3$ is 
injective, we may conclude that $\nabla_c\Sigma=\Phi\Omega_c$. By the Bianchi 
identity, or by direct calculation, one readily verifies that 
$\nabla_a\Phi=0$. It follows that 
\[\nabla_a(-\Phi\Sigma/4)=-\Phi^2\Omega_a/4=\Omega_a,\]
as required. 

It remains to prove the exactness of~(\ref{zero}). Notice from
(\ref{curvatureskewtracefreetractors}) that $\Phi$ already vanishes on
$\Lambda_{\perp,0}^2{\mathbb{U}}$, which means that our connection
(\ref{skewtracefreetractors}) is actually flat
on~$\Lambda_{\perp.0}^2{\mathbb{U}}$. As ${\mathbb{CP}}_n$ is simply-connected,
the vector bundle $\Lambda_{\perp,0}^2{\mathbb{U}}$ is trivialised by this
connection and the exactness of (\ref{zero}) follows immediately from the 
corresponding uncoupled Theorem~\ref{desired}.
\end{proof}

The complete proof in case $\ell=2$ may seem rather complicated and, indeed,
the detailed analysis is necessarily severe. However, as we shall see in the
following section, the general argument can be given rather cleanly. In
particular, the awkward Lemma~\ref{purealgebra} can be formulated and finessed
by means of suitable Lie algebra cohomology.

\subsection{The case $\ell\geq 2$}\label{generalell} The following discussion
is a strict generalisation of the case $\ell=2$ given in~\S\ref{ellequalstwo}
above. Firstly we generalise the bundle $\Lambda_\perp^2{\mathbb{U}}$ and its
connection constructed from~(\ref{newtractors}). Recall that the natural
structure group for ${\mathbb{U}}$ is ${\mathrm{Sp}}(2(n+1),{\mathbb{R}})$ and
so we may form an induced bundle with connection for any irreducible
representation thereof. In particular, the bundle 
$\Lambda_\perp^2{\mathbb{U}}$ arises from the representation 
\[\begin{picture}(90,10)
\put(0,0){\line(1,0){50}}
\put(60,0){\makebox(0,0){$\cdots$}}
\put(70,0){\line(1,0){5}}
\put(75,1){\line(1,0){15}}
\put(75,-1){\line(1,0){15}}
\put(0,0){\makebox(0,0){$\bullet$}}
\put(0,7){\makebox(0,0){\scriptsize$0$}}
\put(15,0){\makebox(0,0){$\bullet$}}
\put(15,7){\makebox(0,0){\scriptsize$1$}}
\put(30,0){\makebox(0,0){$\bullet$}}
\put(30,7){\makebox(0,0){\scriptsize$0$}}
\put(45,0){\makebox(0,0){$\bullet$}}
\put(45,7){\makebox(0,0){\scriptsize$0$}}
\put(75,0){\makebox(0,0){$\bullet$}}
\put(75,7){\makebox(0,0){\scriptsize$0$}}
\put(90,0){\makebox(0,0){$\bullet$}}
\put(90,7){\makebox(0,0){\scriptsize$0$}}
\put(82,0){\makebox(0,0){$\langle$}}
\end{picture}\qquad\mbox{($n+1$ nodes)}\]
and, more generally, let us consider the bundle $Y_\perp^{\ell-1}{\mathbb{U}}$
induced by
\begin{equation}\label{YbundleonCPn}\begin{picture}(100,10)
\put(0,0){\line(1,0){60}}
\put(70,0){\makebox(0,0){$\cdots$}}
\put(80,0){\line(1,0){5}}
\put(85,1){\line(1,0){15}}
\put(85,-1){\line(1,0){15}}
\put(0,0){\makebox(0,0){$\bullet$}}
\put(0,7){\makebox(0,0){\scriptsize$0$}}
\put(20,0){\makebox(0,0){$\bullet$}}
\put(20,7){\makebox(0,0){\scriptsize$\ell-1$}}
\put(40,0){\makebox(0,0){$\bullet$}}
\put(40,7){\makebox(0,0){\scriptsize$0$}}
\put(55,0){\makebox(0,0){$\bullet$}}
\put(55,7){\makebox(0,0){\scriptsize$0$}}
\put(85,0){\makebox(0,0){$\bullet$}}
\put(85,7){\makebox(0,0){\scriptsize$0$}}
\put(100,0){\makebox(0,0){$\bullet$}}
\put(100,7){\makebox(0,0){\scriptsize$0$}}
\put(92,0){\makebox(0,0){$\langle$}}
\end{picture}\;\raisebox{-2pt}{.}\end{equation}
As tensor bundles, these are quite complicated, e.g.
\[Y_\perp^1{\mathbb{U}}=\Lambda_\perp^2{\mathbb{U}}=
\begin{picture}(6,6)
\put(0,0){\line(1,0){6}}
\put(0,0){\line(0,1){6}}
\put(0,6){\line(1,0){6}}
\put(6,0){\line(0,1){6}}
\end{picture}\oplus
\begin{picture}(6,6)
\put(0,-3){\line(1,0){6}}
\put(0,-3){\line(0,1){12}}
\put(0,3){\line(1,0){6}}
\put(6,-3){\line(0,1){12}}
\put(0,9){\line(1,0){6}}
\end{picture}\oplus
\begin{picture}(6,6)
\put(0,0){\line(1,0){6}}
\put(0,0){\line(0,1){6}}
\put(0,6){\line(1,0){6}}
\put(6,0){\line(0,1){6}}
\end{picture}
\qquad\qquad Y_\perp^2{\mathbb{U}}=
\begin{picture}(12,6)
\put(0,0){\line(1,0){12}}
\put(0,0){\line(0,1){6}}
\put(0,6){\line(1,0){12}}
\put(6,0){\line(0,1){6}}
\put(12,0){\line(0,1){6}}
\end{picture}\oplus
\begin{picture}(12,6)
\put(0,-3){\line(1,0){6}}
\put(0,-3){\line(0,1){12}}
\put(0,3){\line(1,0){12}}
\put(6,-3){\line(0,1){12}}
\put(0,9){\line(1,0){12}}
\put(12,3){\line(0,1){6}}
\end{picture}\oplus\!\!\!
\raisebox{2pt}{$\begin{array}c\begin{picture}(12,12)
\put(0,-3){\line(1,0){12}}
\put(0,-3){\line(0,1){12}}
\put(0,3){\line(1,0){12}}
\put(6,-3){\line(0,1){12}}
\put(12,-3){\line(0,1){12}}
\put(0,9){\line(1,0){12}}
\end{picture}\\ \\
\begin{picture}(12,6)
\put(0,0){\line(1,0){12}}
\put(0,0){\line(0,1){6}}
\put(0,6){\line(1,0){12}}
\put(6,0){\line(0,1){6}}
\put(12,0){\line(0,1){6}}
\end{picture}\end{array}$}\!\!\!\oplus
\begin{picture}(12,6)
\put(0,-3){\line(1,0){6}}
\put(0,-3){\line(0,1){12}}
\put(0,3){\line(1,0){12}}
\put(6,-3){\line(0,1){12}}
\put(0,9){\line(1,0){12}}
\put(12,3){\line(0,1){6}}
\end{picture}\oplus
\begin{picture}(12,6)
\put(0,0){\line(1,0){12}}
\put(0,0){\line(0,1){6}}
\put(0,6){\line(1,0){12}}
\put(6,0){\line(0,1){6}}
\put(12,0){\line(0,1){6}}
\end{picture}\;,\]
but we shall not need to know the details. The curvature of the induced
connection on $Y_\perp^{\ell-1}{\mathbb{U}}$ is given by
\begin{equation}\label{formofcurvature}
(\nabla_a\nabla_b-\nabla_b\nabla_a)\Sigma=2J_{ab}\Psi\Sigma,\end{equation}
where $\Psi\in\End(Y_\perp^{\ell-1}{\mathbb{U}})$ is induced by
$\Phi\in\End({\mathbb{U}})$ defined by equations (\ref{curvatureofU})
and~(\ref{symplecticEinstein}). The form of the curvature
(\ref{formofcurvature}) is all we know in order to proceed with a rather
general construction as follows. We shall be mimicking the construction of the
Bernstein-Gelfand-Gelfand complex on projective space given in~\cite{eg}.
For simplicity let us suppose that $n\geq 3$, postponing the case $n=2$ for
later discussion. It is clear that the complex (\ref{nextRScomplex}) can be
naturally coupled with $Y_\perp^{\ell-1}{\mathbb{U}}$ to yield a complex
\begin{equation}\label{coupledRS}
Y_\perp^{\ell-1}{\mathbb{U}}\xrightarrow{\,\nabla\,}
\Lambda^1\otimes Y_\perp^{\ell-1}{\mathbb{U}}
\xrightarrow{\,\nabla_\perp\,}
\Lambda_\perp^2\otimes Y_\perp^{\ell-1}{\mathbb{U}}
\xrightarrow{\,\nabla_\perp\,}
\Lambda_\perp^3\otimes Y_\perp^{\ell-1}{\mathbb{U}}\end{equation}
and we maintain that this complex is naturally filtered. This is because, 
by~(\ref{newtractors}), the same is evidently true of the
${\mathbb{U}}$-coupled complex
\[\begin{array}{ccccccc}
{\mathbb{U}}&\xrightarrow{\,\nabla\,}&
\Lambda^1\otimes{\mathbb{U}}
&\xrightarrow{\,\nabla_\perp\,}&
\Lambda_\perp^2\otimes{\mathbb{U}}
&\xrightarrow{\,\nabla_\perp\,}&
\Lambda_\perp^3\otimes{\mathbb{U}}\\
\|&&\|&&\|&&\|\\
\begin{array}c\Lambda^0\\ \oplus\\ \Lambda^1\\ \oplus\\ \Lambda^0\end{array}
&\begin{array}c \\ 
\begin{picture}(0,0)\put(-9,-6){\vector(3,2){18}}\end{picture}\\ 
\\ 
\begin{picture}(0,0)\put(-9,-6){\vector(3,2){18}}\end{picture}\\ 
\\ 
\end{array}&
\begin{array}c\Lambda^1\\ \oplus\\ \Lambda^1\otimes\Lambda^1\\ \oplus\\ 
\Lambda^1\end{array}
&\begin{array}c \\ 
\begin{picture}(0,0)\put(-9,-6){\vector(3,2){18}}\end{picture}\\ 
\\ 
\begin{picture}(0,0)\put(-9,-6){\vector(3,2){18}}\end{picture}\\ 
\\ 
\end{array}&
\begin{array}c\Lambda_\perp^2\\ \oplus\\ \Lambda_\perp^2\otimes\Lambda^1\\ 
\oplus\\ \Lambda_\perp^2\end{array}
&\begin{array}c \\ 
\begin{picture}(0,0)\put(-9,-6){\vector(3,2){18}}\end{picture}\\ 
\\ 
\begin{picture}(0,0)\put(-9,-6){\vector(3,2){18}}\end{picture}\\ 
\\ 
\end{array}&
\begin{array}c\Lambda_\perp^3\\ \oplus\\ \Lambda_\perp^3\otimes\Lambda^1\\ 
\oplus\\ \Lambda_\perp^3\end{array}
\end{array}\] 
and the filtration on (\ref{coupledRS}) is inherited therefrom. Writing
\[Y_\perp^{\ell-1}{\mathbb{U}}={\mathbb{V}}=
{\mathbb{V}}_0\oplus{\mathbb{V}}_1\oplus{\mathbb{V}}_2\oplus{\mathbb{V}}_3
\oplus\cdots\oplus{\mathbb{V}}_N\]
for the associated graded vector bundle, the $E_0$-level of the resulting
spectral sequence has the form
\begin{equation}\label{E0level}\raisebox{-60pt}{\begin{picture}(300,100)
\put(0,0){\vector(1,0){100}}
\put(0,0){\vector(0,1){50}}
\put(92,5){\makebox(0,0){$p$}}
\put(5,42){\makebox(0,0){$q$}}
\put(20,90){\makebox(0,0){${\mathbb{V}}_0$}}
\put(70,60){\makebox(0,0){${\mathbb{V}}_1$}}
\put(120,30){\makebox(0,0){${\mathbb{V}}_2$}}
\put(140,15){\makebox(0,0){$\ddots$}}
\put(70,90){\makebox(0,0){$\Lambda^1\otimes{\mathbb{V}}_0$}}
\put(120,60){\makebox(0,0){$\Lambda^1\otimes{\mathbb{V}}_1$}}
\put(170,30){\makebox(0,0){$\Lambda^1\otimes{\mathbb{V}}_2$}}
\put(190,15){\makebox(0,0){$\ddots$}}
\put(120,90){\makebox(0,0){$\Lambda_\perp^2\otimes{\mathbb{V}}_0$}}
\put(170,60){\makebox(0,0){$\Lambda_\perp^2\otimes{\mathbb{V}}_1$}}
\put(220,30){\makebox(0,0){$\Lambda_\perp^2\otimes{\mathbb{V}}_2$}}
\put(240,15){\makebox(0,0){$\ddots$}}
\put(170,90){\makebox(0,0){$\Lambda_\perp^3\otimes{\mathbb{V}}_0$}}
\put(220,60){\makebox(0,0){$\Lambda_\perp^3\otimes{\mathbb{V}}_1$}}
\put(270,30){\makebox(0,0){$\Lambda_\perp^3\otimes{\mathbb{V}}_2$}}
\put(290,15){\makebox(0,0){$\ddots$}}
\put(70,75){\makebox(0,0){$\uparrow\scriptstyle\partial$}}
\put(124,75){\makebox(0,0){$\uparrow\scriptstyle\partial_\perp$}}
\put(174,75){\makebox(0,0){$\uparrow\scriptstyle\partial_\perp$}}
\put(120,45){\makebox(0,0){$\uparrow\scriptstyle\partial$}}
\put(174,45){\makebox(0,0){$\uparrow\scriptstyle\partial_\perp$}}
\put(224,45){\makebox(0,0){$\uparrow\scriptstyle\partial_\perp$}}
\end{picture}}\end{equation}
where all differentials are vector bundle homomorphisms. The key point is that
we can identify much of the $E_1$-level explicitly. Take, for example, the 
homomorphism $\partial:{\mathbb{V}}_1\to\Lambda^1\otimes{\mathbb{V}}_0$. It is 
induced by the identity mapping 
$\Lambda^1={\mathbb{U}}_1\to\Lambda^1\otimes{\mathbb{U}}_0=\Lambda^1$ and thus
may be identified as the canonical inclusion  
\[{\mathbb{V}}_1=\begin{picture}(36,6)
\put(0,-3){\line(1,0){6}}
\put(0,-3){\line(0,1){12}}
\put(0,3){\line(1,0){36}}
\put(6,-3){\line(0,1){12}}
\put(0,9){\line(1,0){36}}
\put(12,3){\line(0,1){6}}
\put(30,3){\line(0,1){6}}
\put(36,3){\line(0,1){6}}
\put(22,5.5){\makebox(0,0){$\cdots$}}
\end{picture}\hookrightarrow
\begin{picture}(6,6)
\put(0,0){\line(0,1){6}}
\put(0,0){\line(1,0){6}}
\put(6,0){\line(0,1){6}}
\put(0,6){\line(1,0){6}}
\end{picture}\otimes\underbrace{\begin{picture}(36,6)
\put(0,0){\line(0,1){6}}
\put(0,0){\line(1,0){36}}
\put(6,0){\line(0,1){6}}
\put(0,6){\line(1,0){36}}
\put(12,0){\line(0,1){6}}
\put(30,0){\line(0,1){6}}
\put(36,0){\line(0,1){6}}
\put(22,2.5){\makebox(0,0){$\cdots$}}
\end{picture}}_{\ell-1\ \mathrm{boxes}}=\Lambda^1\otimes{\mathbb{V}}_0\]
with quotient $\bigodot^\ell\!\Lambda^1$. As a more subtle example, when 
$\ell=2$ we have ${\mathbb{V}}=\Lambda_\perp^1{\mathbb{U}}$ with 
${\mathbb{V}}_0\oplus{\mathbb{V}}_1\oplus{\mathbb{V}}_2=
\Lambda^1\oplus\Lambda^2\oplus\Lambda^1$ and 
\[\begin{array}{ccccc}
\Lambda^1\otimes{\mathbb{V}}_2&\xrightarrow{\,\partial_\perp
\,}&
\Lambda_\perp^2\otimes{\mathbb{V}}_1&\xrightarrow{\,\partial_\perp\,}&
\Lambda_\perp^3\otimes{\mathbb{V}}_0\\
\|&&\|&&\|\\
\Lambda^1\otimes\Lambda^1&\to&\Lambda_\perp^2\otimes\Lambda^2
&\to&\Lambda_\perp^3\otimes\Lambda^1
\end{array}\]
whose cohomology is actually the subject of Lemma~\ref{purealgebra}, being
identified there as
\[H=\{X_{abcd}\mbox{ s.t.\ }X_{abcd}=X_{[ab][cd]}\mbox{ and }X_{[abc]d}=0
\mbox{ and }J^{ab}X_{abcd}=0\}\]
corresponding to the irreducible representation 
\[\begin{picture}(12,12)
\put(0,-3){\line(1,0){12}}
\put(0,-3){\line(0,1){12}}
\put(0,3){\line(1,0){12}}
\put(6,-3){\line(0,1){12}}
\put(12,-3){\line(0,1){12}}
\put(0,9){\line(1,0){12}}
\put(15,-2){\makebox(0,0){$\scriptstyle\perp$}}
\end{picture}\;=\;\begin{picture}(75,10)
\put(0,0){\line(1,0){35}}
\put(45,0){\makebox(0,0){$\cdots$}}
\put(55,0){\line(1,0){5}}
\put(60,1){\line(1,0){15}}
\put(60,-1){\line(1,0){15}}
\put(0,0){\makebox(0,0){$\bullet$}}
\put(0,7){\makebox(0,0){\scriptsize$0$}}
\put(15,0){\makebox(0,0){$\bullet$}}
\put(15,7){\makebox(0,0){\scriptsize$2$}}
\put(30,0){\makebox(0,0){$\bullet$}}
\put(30,7){\makebox(0,0){\scriptsize$0$}}
\put(60,0){\makebox(0,0){$\bullet$}}
\put(60,7){\makebox(0,0){\scriptsize$0$}}
\put(75,0){\makebox(0,0){$\bullet$}}
\put(75,7){\makebox(0,0){\scriptsize$0$}}
\put(67,0){\makebox(0,0){$\langle$}}
\end{picture}\qquad\mbox{($n$ nodes)}\]
of ${\mathrm{Sp}}(2n,{\mathbb{R}})$. Similarly, the proof of
Theorem~\ref{observation} for $n\geq 3$ boils down to 
\[\begin{array}{crc}
\Lambda_\perp^2\otimes{\mathbb{V}}_2
&\xrightarrow{\hspace*{60pt}\partial_\perp\hspace*{60pt}}
\Lambda_\perp^3\otimes{\mathbb{V}}_1=&
\Lambda_\perp^3\otimes\Lambda^2
\\ \|&&\uparrow\\
\Lambda_\perp^2\otimes\Lambda^1&\ni\rho_{bcd}\mapsto
J_{d[a}\rho_{bc]e}-J_{e[a}\rho_{bc]d}+\rho_{[abc]}J_{de}
\in&\Lambda^3\otimes\Lambda^2
\end{array}\]
being injective. In general, the $E_1$-level of the spectral sequence is 
controlled by Lie algebra cohomology as follows.
\begin{proposition}\label{heisenbergcohomology}
Suppose ${\mathbb{V}}$ is a representation of the Heisenberg algebra
${\mathfrak{h}}_{2n+1}={\mathfrak{g}}_{-2}\oplus{\mathfrak{g}}_{-1}$ of
dimension $2n+1$ for $n\geq 3$ where ${\mathfrak{g}}_{-2}$ denotes the centre
and for $X,Y\in{\mathfrak{g}}_{-1}$, we have $[X,Y]=J(X\wedge Y)$ for a
non-degenerate symplectic form
$J:\Lambda^2{\mathfrak{g}}_{-1}\to{\mathfrak{g}}_{-2}$. Then the Lie algebra 
cohomologies $H^r({\mathfrak{h}}_{2n+1},{\mathbb{V}})$ for $r=0,1,2$ 
may be computed by the complex
\[0\to{\mathbb{V}}\xrightarrow{\,\partial\,}
\Hom({\mathfrak{g}}_{-1},{\mathbb{V}})
\xrightarrow{\,\partial_\perp\,}
\Hom(\Lambda_\perp^2{\mathfrak{g}}_{-1},{\mathbb{V}})
\xrightarrow{\,\partial_\perp\,}
\Hom(\Lambda_\perp^3{\mathfrak{g}}_{-1},{\mathbb{V}}),\]
induced by the action of ${\mathfrak{g}}_{-1}$ on~${\mathbb{V}}$.
\end{proposition}
\begin{proof}
Let us introduce abstract indices in the sense of~\cite{PR1} to write
$\eth_a$ for the action of ${\mathfrak{g}}_{-1}$ on ${\mathbb{V}}$ meaning 
that $Xv=X^a\eth_av$ for $X\in{\mathfrak{g}}_{-1}$. Then 
\[\textstyle v\!\stackrel{\partial}{\mapsto}\!\eth_av\quad
v_a\!\stackrel{\partial_\perp}{\longmapsto}\!
\eth_{[a}v_{b]}-\frac{1}{2n}J^{cd}\eth_cv_dJ_{ab}
\quad
v_{ab}\!\stackrel{\partial_\perp}{\longmapsto}\!
\eth_{[a}v_{bc]}-\frac{1}{n-1}J^{de}\eth_dv_{e[a}J_{bc]}\]
are the explicit formul{\ae} for the differentials of the complex in question.
Let us also write $\eth$ for the action of ${\mathfrak{g}}_{-2}$ 
on~${\mathbb{V}}$. Then to say that ${\mathbb{V}}$ is an 
${\mathfrak{h}}_{2n+1}$-module is precisely that 
\[\eth_a\eth_b-\eth_b\eth_av=2J_{ab}\eth v\quad\forall\,v\in{\mathbb{V}}\]
and the differentials of the usual Koszul complex
$\Lambda^\bullet({\mathfrak{h}}_{2n+1})^*\otimes{\mathbb{V}}$ defining the Lie
algebra cohomology begin with
\[{\mathbb{V}}\ni v\mapsto
\left[\begin{array}c\eth_a v\\ \eth v\end{array}\right]\in
\begin{array}c{\mathfrak{g}}_{-1}^*\otimes{\mathbb{V}}\\ \oplus\\ 
{\mathfrak{g}}_{-2}^*\otimes{\mathbb{V}}\end{array}=
({\mathfrak{h}}_{2n+1})^*\otimes{\mathbb{V}}\]
and continue with
\[\begin{array}{ccc}
\begin{array}c\Lambda^p{\mathfrak{g}}_{-1}^*\otimes{\mathbb{V}}\\ \oplus\\ 
{\mathfrak{g}}_{-2}^*\otimes\Lambda^{p-1}{\mathfrak{g}}_{-1}^*
\otimes{\mathbb{V}}\end{array}&\longrightarrow&
\begin{array}c\Lambda^{p+1}{\mathfrak{g}}_{-1}^*\otimes{\mathbb{V}}\\ \oplus\\ 
{\mathfrak{g}}_{-2}^*\otimes\Lambda^p{\mathfrak{g}}_{-1}^*\otimes{\mathbb{V}}
\end{array}\qquad\mbox{for }p\geq 1\\
\rule{0pt}{22pt}
\left[\begin{array}cv_{ab\cdots cd}\\ w_{ab\cdots c}\end{array}\right]
&\longmapsto&
\left[\begin{array}l\eth_{[a}v_{bc\cdots de]}+(-1)^pJ_{[ab}w_{c\cdots de]}\\ 
\eth_{[a}w_{bc\cdots d]}+(-1)^p\eth v_{abc\cdots d}\end{array}\right]
\end{array}\]
Easy diagram chasing gives the desired result.\end{proof}
\noindent
When $n=2$, there is a replacement for Proposition~\ref{heisenbergcohomology}
as follows.
\begin{proposition}\label{replacementproposition}
Suppose ${\mathbb{V}}$ is a representation of the Heisenberg algebra
${\mathfrak{h}}_{5}={\mathfrak{g}}_{-2}\oplus{\mathfrak{g}}_{-1}$ of
dimension~$5$. Then the Lie algebra 
cohomologies $H^r({\mathfrak{h}}_5,{\mathbb{V}})$ for $r=0,1,2$ 
may be computed by the complex
\[0\to{\mathbb{V}}\xrightarrow{\,\partial\,}
{\mathfrak{g}}_{-1}^*\otimes{\mathbb{V}}
\xrightarrow{\,\partial_\perp\,}
\Lambda_\perp^2{\mathfrak{g}}_{-1}^*\otimes{\mathbb{V}}
\xrightarrow{\,\partial_\perp^{(2)}\,}
{\mathfrak{g}}_{-2}^*\otimes\Lambda_\perp^2{\mathfrak{g}}_{-1}^*\otimes
{\mathbb{V}}\]
where, adopting the notation from the proof of
Proposition~\ref{heisenbergcohomology}, the linear transformation
$\partial_\perp^{(2)}\!\!:
\Lambda_\perp^2{\mathfrak{g}}_{-1}^*\otimes{\mathbb{V}}\to
{\mathfrak{g}}_{-2}^*\otimes\Lambda_\perp^2{\mathfrak{g}}_{-1}^*\otimes
{\mathbb{V}}$ is given by 
\[v_{ab}\longmapsto J^{cd}\eth_c\eth_{[a}v_{b]d}+3\eth v_{ab}.\]
\end{proposition}
\begin{proof}
This is what emerges by following the proof of 
Proposition~\ref{heisenbergcohomology}. The only difference when $n=2$ is that 
$\Lambda_\perp^3{\mathfrak{g}}_{-1}^*=0$ and so there is one extra step in the 
resulting diagram chase.
\end{proof}

Looking back at the $E_0$-level (\ref{E0level}) of our spectral sequence, we
see that for $n\geq 3$, Proposition~\ref{heisenbergcohomology} is exactly what
we need to identify much of the $E_1$-level provided we are able to identify
the Lie algebra cohomology $H^r({\mathfrak{h}}_{2n+1},{\mathbb{V}})$ for
$r=0,1,2$. Kostant's Theorem~\cite{k} provides such an identification
\[\begin{array}{rcl}
H^0({\mathfrak{h}}_{2n+1},Y_\perp^{\ell-1}{\mathbb{U}})&=&
\enskip\begin{picture}(80,10)
\put(0,0){\line(1,0){35}}
\put(50,0){\makebox(0,0){$\cdots$}}
\put(60,0){\line(1,0){5}}
\put(65,1){\line(1,0){15}}
\put(65,-1){\line(1,0){15}}
\put(0,0){\makebox(0,0){$\bullet$}}
\put(0,7){\makebox(0,0){\scriptsize$\ell-1$}}
\put(20,0){\makebox(0,0){$\bullet$}}
\put(20,7){\makebox(0,0){\scriptsize$0$}}
\put(35,0){\makebox(0,0){$\bullet$}}
\put(35,7){\makebox(0,0){\scriptsize$0$}}
\put(65,0){\makebox(0,0){$\bullet$}}
\put(65,7){\makebox(0,0){\scriptsize$0$}}
\put(80,0){\makebox(0,0){$\bullet$}}
\put(80,7){\makebox(0,0){\scriptsize$0$}}
\put(72,0){\makebox(0,0){$\langle$}}
\end{picture}\qquad\mbox{($n$ nodes)}\\
\rule{0pt}{14pt}H^1({\mathfrak{h}}_{2n+1},Y_\perp^{\ell-1}{\mathbb{U}})&=&
\enskip\begin{picture}(75,10)
\put(0,0){\line(1,0){35}}
\put(45,0){\makebox(0,0){$\cdots$}}
\put(55,0){\line(1,0){5}}
\put(60,1){\line(1,0){15}}
\put(60,-1){\line(1,0){15}}
\put(0,0){\makebox(0,0){$\bullet$}}
\put(0,7){\makebox(0,0){\scriptsize$\ell$}}
\put(15,0){\makebox(0,0){$\bullet$}}
\put(15,7){\makebox(0,0){\scriptsize$0$}}
\put(30,0){\makebox(0,0){$\bullet$}}
\put(30,7){\makebox(0,0){\scriptsize$0$}}
\put(60,0){\makebox(0,0){$\bullet$}}
\put(60,7){\makebox(0,0){\scriptsize$0$}}
\put(75,0){\makebox(0,0){$\bullet$}}
\put(75,7){\makebox(0,0){\scriptsize$0$}}
\put(67,0){\makebox(0,0){$\langle$}}
\end{picture}\\
\rule{0pt}{14pt}H^2({\mathfrak{h}}_{2n+1},Y_\perp^{\ell-1}{\mathbb{U}})&=&
\enskip\begin{picture}(75,10)
\put(0,0){\line(1,0){35}}
\put(45,0){\makebox(0,0){$\cdots$}}
\put(55,0){\line(1,0){5}}
\put(60,1){\line(1,0){15}}
\put(60,-1){\line(1,0){15}}
\put(0,0){\makebox(0,0){$\bullet$}}
\put(0,7){\makebox(0,0){\scriptsize$0$}}
\put(15,0){\makebox(0,0){$\bullet$}}
\put(15,7){\makebox(0,0){\scriptsize$\ell$}}
\put(30,0){\makebox(0,0){$\bullet$}}
\put(30,7){\makebox(0,0){\scriptsize$0$}}
\put(60,0){\makebox(0,0){$\bullet$}}
\put(60,7){\makebox(0,0){\scriptsize$0$}}
\put(75,0){\makebox(0,0){$\bullet$}}
\put(75,7){\makebox(0,0){\scriptsize$0$}}
\put(67,0){\makebox(0,0){$\langle$}}
\end{picture}\end{array}\]
as ${\mathrm{Sp}}(2n,{\mathbb{R}})$-modules as well as the exact locations of
the corresponding induced bundles in the $E_1$-level
\begin{equation}\label{E1level}\raisebox{-45pt}{\begin{picture}(300,80)
\put(0,0){\vector(1,0){100}}
\put(0,0){\vector(0,1){50}}
\put(92,5){\makebox(0,0){$p$}}
\put(5,42){\makebox(0,0){$q$}}
\put(40,70){\makebox(0,0){$\bigodot^{\ell-1}\!\Lambda^1$}}
\put(70,68){\makebox(0,0){$\to$}}
\put(94,70){\makebox(0,0){$\bigodot^{\ell}\!\Lambda^1$}}
\put(125,70){\makebox(0,0){$0$}}
\put(150,70){\makebox(0,0){$\ast$}}
\put(150,55){\makebox(0,0){$0$}}
\put(175,55){\makebox(0,0){$\ast$}}
\put(175,40){\makebox(0,0){$\ddots$}}
\put(202,40){\makebox(0,0){$\ddots$}}
\put(197,25){\makebox(0,0){$Y_\perp^\ell$}}
\put(225,25){\makebox(0,0){$\ast$}}
\put(225,10){\makebox(0,0){$\ddots$}}
\put(252,10){\makebox(0,0){$\ddots$}}
\end{picture}}\end{equation}
where $Y_\perp^\ell$ arises as the cohomology of
\[\Lambda^1\otimes{\mathbb{V}}_\ell\xrightarrow{\,\partial_\perp\,}
\Lambda_\perp^2\otimes{\mathbb{V}}_{\ell-1}\xrightarrow{\,\partial_\perp\,}
\Lambda_\perp^3\otimes{\mathbb{V}}_{\ell-2}.\]
This is the right location to be the target of a differential 
$\bigodot^\ell\!\Lambda^1\to Y_\perp^\ell$ at the $E_\ell$-level. We conclude 
immediately that there is a complex
\[\textstyle\bigodot^{\ell-1}\!\Lambda^1\xrightarrow{\,\nabla\,}
\bigodot^{\ell}\!\Lambda^1\xrightarrow{\,\nabla_\perp^{(\ell)}\,}
Y_\perp^\ell\] 
whose cohomology is the same as that of the original complex
\[Y_\perp^{\ell-1}{\mathbb{U}}\xrightarrow{\,\nabla\,}
\Lambda^1\otimes Y_\perp^{\ell-1}{\mathbb{U}}
\xrightarrow{\,\nabla_\perp\,}
\Lambda_\perp^2\otimes Y_\perp^{\ell-1}{\mathbb{U}}.\]
This is true both locally (which confirms abstractly~\cite{p} that
$\nabla_\perp^{(\ell)}$ is a differential operator) and globally, which is what
we shall use to prove the sufficiency of the condition given in
Theorem~\ref{necessaryonCPn} as follows.
\begin{theorem}\label{sufficiencyonCPn}
Suppose $\omega_{abc\cdots d}$ is a globally defined smooth symmetric 
\mbox{${\ell}$-tensor}
on~${\mathbb{CP}}_n$ for $n\geq2$. Suppose that
$\nabla_\perp^{(\ell)}(\omega_{ab\cdots d})=0$, where $\nabla_\perp^{(\ell)}$
is the differential operator of
Theorem~\ref{necessaryonCPn} defined as the
composition~{\rm(\ref{composition})}. Then there is a smooth symmetric 
$(\ell-1)$-tensor $\phi_{bc\cdots d}$ on ${\mathbb{CP}}_n$ such that
$\omega_{abc\cdots d}=\nabla_{(a}\phi_{bc\cdots d)}$.
\end{theorem}
\begin{proof}
For $n\geq 3$, there remain just two facts to verify. The first is that, as our
notation already indicates, the differential operator $\nabla_\perp^{(\ell)}$
arising at the $E_{\ell}$-level of the spectral sequence of the filtered
complex (\ref{coupledRS}) coincides with the composition~(\ref{composition}). 
The second is that 
\begin{equation}\label{secondfact}
\Gamma({\mathbb{CP}}_n,Y_\perp^{\ell-1}{\mathbb{U}})\!\xrightarrow{\nabla}\!
\Gamma({\mathbb{CP}}_n,\Lambda^1\otimes Y_\perp^{\ell-1}{\mathbb{U}})
\!\xrightarrow{\nabla_\perp}\!
\Gamma({\mathbb{CP}}_n,\Lambda_\perp^2\otimes Y_\perp^{\ell-1}{\mathbb{U}})
\end{equation}
is exact. The first of these two facts follows by comparing the spectral
sequence constructions of $\nabla_\perp^{(\ell)}$ on ${\mathbb{CP}}_n$ and
$\nabla^{(\ell)}$ on~${\mathbb{RP}}_n$ for the model embeddings
$\iota:{\mathbb{RP}}_n\hookrightarrow{\mathbb{CP}}_n$. Recall, as in
Proposition~\ref{Two}, that each $\iota$ is totally geodesic and Lagrangian. If
$\Sigma=(\sigma,\mu,\rho)$ is a section of
${\mathbb{U}}\equiv\Lambda^0\oplus\Lambda^1\oplus\Lambda^0$ on
${\mathbb{CP}}_n$ then we can define its `pull-back' $\iota^*\Sigma$ to
${\mathbb{RP}}_n$ to be the section $(\iota^*\sigma,\iota^*\mu)$ of the bundle
${\mathbb{T}}\equiv\Lambda^0\oplus\Lambda^1$ on~${\mathbb{RP}}_n$. {From} the
formul{\ae} (\ref{newtractors}) and (\ref{standardtractorconnection}) for the
connections on ${\mathbb{U}}$ and~${\mathbb{T}}$, we see that pull-back
intertwines these connections,
i.e.~$\iota^*(\nabla\Sigma)=\nabla(\iota^*\Sigma)$. The same is therefore true 
of pull-back from 
\[Y_\perp^{\ell-1}{\mathbb{U}}=
\underbrace{\begin{picture}(54,12)
\put(0,0){\line(1,0){54}}
\put(0,6){\line(1,0){54}}
\put(0,12){\line(1,0){54}}
\put(0,0){\line(0,1){12}}
\put(6,0){\line(0,1){12}}
\put(12,0){\line(0,1){12}}
\put(18,0){\line(0,1){12}}
\put(30,3){\makebox(0,0){$\cdots$}}
\put(30,9){\makebox(0,0){$\cdots$}}
\put(42,0){\line(0,1){12}}
\put(48,0){\line(0,1){12}}
\put(54,0){\line(0,1){12}}
\end{picture}}_{\mbox{\scriptsize$\ell-1$ columns}}{}_\perp
\mbox{\LARGE${\mathbb{U}}$}\quad\mbox{to}\quad
\underbrace{\begin{picture}(54,12)
\put(0,0){\line(1,0){54}}
\put(0,6){\line(1,0){54}}
\put(0,12){\line(1,0){54}}
\put(0,0){\line(0,1){12}}
\put(6,0){\line(0,1){12}}
\put(12,0){\line(0,1){12}}
\put(18,0){\line(0,1){12}}
\put(30,3){\makebox(0,0){$\cdots$}}
\put(30,9){\makebox(0,0){$\cdots$}}
\put(42,0){\line(0,1){12}}
\put(48,0){\line(0,1){12}}
\put(54,0){\line(0,1){12}}
\end{picture}}_{\mbox{\scriptsize$\ell-1$ columns}}\,
\mbox{\LARGE${\mathbb{T}}$}.\]
Hence, the resulting operators
\[\textstyle\bigodot^{\ell}\!\Lambda^1\xrightarrow{\,\nabla_\perp^{(\ell)}\,}
Y_\perp^\ell\enskip\mbox{on }{\mathbb{CP}}_n\quad\mbox{and}\quad
\bigodot^{\ell}\!\Lambda^1\xrightarrow{\,\nabla^{(\ell)}\,}Y^\ell
\enskip\mbox{on }{\mathbb{RP}}_n\] 
are also related by pull-back. As this is true for every model embedding
$\iota:{\mathbb{RP}}_n\hookrightarrow{\mathbb{CP}}_n$,
Corollary~\ref{linearalgebra} now shows that $\nabla_\perp^{(\ell)}$ is
characterised by this property. By construction, (\ref{composition}) also has 
this property and so the two operators agree, as required.

Now we must show that (\ref{secondfact}) is exact. The following reasoning
applies to any bundle ${\mathbb{V}}$ induced from ${\mathbb{U}}$ by an
irreducible representation of ${\mathrm{Sp}}(2(n+1),{\mathbb{R}})$ including
$Y_\perp^{\ell-1}{\mathbb{U}}$ (recall that it is induced
by~(\ref{YbundleonCPn})). In particular, when applied to 
${\mathbb{V}}=\Lambda_\perp^2{\mathbb{U}}$ it puts the reasoning in the proof
of Theorem~\ref{prototypesufficiencyonCPn} in proper context. The endomorphism
\[\left[\begin{array}c\sigma\\ \mu_c\\ \rho\end{array}\right]
\stackrel{\Phi}{\longmapsto}
\left[\begin{array}c\rho\\ J_c{}^d\mu_d\\ -\sigma\end{array}\right]\]
of ${\mathbb{U}}$ defined by (\ref{symplecticEinstein}) is preserved by the
connection on~${\mathbb{U}}$. Evidently, it also satisfies
$\Phi^2=-{\mathrm{Id}}$. Finally, we compute
\[\langle\Phi(\sigma,\mu_a,\rho),(\tilde\sigma,\tilde\mu_b,\tilde\rho)\rangle=
\sigma\tilde\sigma+g^{ab}\mu_a\tilde\mu_b+\rho\tilde\rho\]
where $\langle\enskip,\enskip\rangle$ is the skew form on ${\mathbb{U}}$
defined by~(\ref{symplecticreduction}) and notice that this is symmetric and of
Lorentzian signature. Hence, the structure group for ${\mathbb{U}}$ naturally
reduces from ${\mathrm{Sp}}(2(n+1),{\mathbb{R}})$ to ${\mathrm{SU}}(2n+1,1)$
with $\Phi$ providing the complex structure. Accordingly, the
decomposition (\ref{decomposition}) may be seen as follows. The complexified 
bundles split in familiar fashion as
\[{\mathbb{CU}}=\Lambda^{1,0}{\mathbb{U}}\oplus\Lambda^{0,1}{\mathbb{U}}\]
and 
\[\Lambda^2{\mathbb{CU}}=\Lambda^{2,0}{\mathbb{U}}\oplus
\left(\Lambda_\perp^{1,1}{\mathbb{U}}\oplus\Lambda^0{\mathbb{U}}\right)\oplus
\Lambda^{0,2}{\mathbb{U}}\]
as complex eigenspaces under the action of $\Phi$. Then (\ref{decomposition}) 
is simply the real counterpart of the complex decomposition
\[\Lambda_\perp^2{\mathbb{CU}}=
\Lambda_\perp^{1,1}{\mathbb{U}}\oplus
\left(\Lambda^{2,0}{\mathbb{U}}\oplus\Lambda^{0,2}{\mathbb{U}}\right).\]
For our purposes, a sufficient counterpart to (\ref{decomposition}) in general 
is to write
\[{\mathbb{V}}={\mathbb{V}}_0\oplus{\mathbb{V}}_{\not=0}\]
where ${\mathbb{V}}_0\equiv\ker\Psi:{\mathbb{V}}\to{\mathbb{V}}$ (and,
following~(\ref{formofcurvature}), we are writing $\Psi$ for endomorphism of
${\mathbb{V}}$ induced by $\Phi\in\End({\mathbb{U}})$) and
${\mathbb{V}}_{\not=0}$ is such that ${\mathbb{CV}}_{\not=0}$ collects all the
non-zero eigenspaces of $\Psi$. In particular, notice that
$\Psi|_{{\mathbb{V}}_{\not=0}}:{\mathbb{V}}_{\not=0}\to{\mathbb{V}}_{\not=0}$
is invertible. Now we are in a position to show that
\begin{equation}\label{Vexact}
\Gamma({\mathbb{CP}}_n,{\mathbb{V}})\xrightarrow{\,\nabla\,}
\Gamma({\mathbb{CP}}_n,\Lambda^1\otimes{\mathbb{V}})
\xrightarrow{\,\nabla_\perp\,}
\Gamma({\mathbb{CP}}_n,\Lambda_\perp^2\otimes{\mathbb{V}})\end{equation}
in general, and hence (\ref{secondfact}) in particular, is exact. As in the 
proof of Theorem~\ref{prototypesufficiencyonCPn}, this breaks into two cases
\[\begin{array}{rcccl}
\Gamma({\mathbb{CP}}_n,{\mathbb{V}}_0)&\xrightarrow{\,\nabla\,}&
\Gamma({\mathbb{CP}}_n,\Lambda^1\otimes{\mathbb{V}}_0)
&\xrightarrow{\,\nabla_\perp\,}&
\Gamma({\mathbb{CP}}_n,\Lambda_\perp^2\otimes{\mathbb{V}}_0)\\
\Gamma({\mathbb{CP}}_n,{\mathbb{V}}_{\not=0})&\xrightarrow{\,\nabla\,}&
\Gamma({\mathbb{CP}}_n,\Lambda^1\otimes{\mathbb{V}}_{\not=0})
&\xrightarrow{\,\nabla_\perp\,}&
\Gamma({\mathbb{CP}}_n,\Lambda_\perp^2\otimes{\mathbb{V}}_{\not=0}),
\end{array}\]
the counterparts of (\ref{zero}) and~(\ref{minusfour}). The first of these
complexes is exact as a consequence of Theorem~\ref{desired} coupled with the
flat connection $\nabla|_{{\mathbb{V}}_0}$. The curvature of 
$\nabla|_{{\mathbb{V}}_{\not=0}}$ has the form (\ref{formofcurvature}) with 
$\Psi\in\End({\mathbb{V}}_{\not=0})$ crucially being invertible. Exactness of 
the corresponding complex is established as follows. Suppose 
$\Omega\in\Gamma({\mathbb{CP}}_n,\Lambda^1\otimes{\mathbb{V}}_{\not=0})$ 
satisfies $\nabla_\perp\Omega=0$. Precisely, this means that 
\[\nabla_{[a}\Omega_{b]}=J_{ab}\Sigma\quad\mbox{for some }
\Sigma\in\Gamma({\mathbb{CP}}_n,{\mathbb{V}}_{\not=0}).\]
Differentiating again, we find that  
\[J_{[ab}\nabla_{c]}\Sigma=\nabla_{[a}(J_{bc]}\Sigma)
=\nabla_{[a}\nabla_b\Omega_{c]}=J_{[ab}\Psi\Omega_{c]},\]
the last equality being a consequence of (\ref{formofcurvature}) as applied to
the vector bundle~${\mathbb{V}}_{\not=0}$. Since $J_{ab}$ is non-degenerate,
we conclude that $\nabla_c\Sigma=\Psi\Omega_c$. But the Bianchi identity 
$\nabla_{[a}(J_{bc]}\Psi)=0$ implies that~$\nabla_a\Psi=0$. Finally, recall 
that $\Psi$ is invertible, whence
\[\nabla_a(\Psi^{-1}\Sigma)=\Psi^{-1}\nabla_a\Sigma=\Psi^{-1}\Psi\Omega_a
=\Omega_a,\]
which is exactly as needed to complete the proof of exactness
of~(\ref{Vexact}).

The proof of Theorem~\ref{sufficiencyonCPn} is complete save for the case $n=2$
for which a modification to the argument is needed as follows. We replace the 
complex (\ref{coupledRS}) by 
\begin{equation}\label{highercoupling}
Y_\perp^{\ell-1}{\mathbb{U}}\xrightarrow{\,\nabla\,}
\Lambda^1\otimes Y_\perp^{\ell-1}{\mathbb{U}}
\xrightarrow{\,\nabla_\perp\,}
\Lambda_\perp^2\otimes Y_\perp^{\ell-1}{\mathbb{U}}
\xrightarrow{\,\nabla_\perp^{(2)}\,}
\Lambda_\perp^2\otimes Y_\perp^{\ell-1}{\mathbb{U}},\end{equation}
where $\nabla_\perp^{(2)}$ is defined by
\[(\nabla_\perp^{(2)}\xi)_{ab}=\nabla_{[a}\mu_{b]}-\Psi\xi_{ab}\quad
\mbox{where }\mu_{[a}J_{bc]}=\nabla_{[a}\xi_{bc]}\]
as a coupled version of the operator 
$d_\perp^{(2)}:\Lambda_\perp^2\to\Lambda_\perp^2$ defined 
by~(\ref{defofdperp2}). It is easy to check that this operator is well-defined
and that (\ref{highercoupling}) is a complex. It is naturally filtered and 
there is spectral sequence whose $E_0$-level has the form
\[\begin{picture}(300,100)
\put(0,0){\vector(1,0){100}}
\put(0,0){\vector(0,1){50}}
\put(92,5){\makebox(0,0){$p$}}
\put(5,42){\makebox(0,0){$q$}}
\put(20,90){\makebox(0,0){${\mathbb{V}}_0$}}
\put(70,60){\makebox(0,0){${\mathbb{V}}_1$}}
\put(120,30){\makebox(0,0){${\mathbb{V}}_2$}}
\put(140,15){\makebox(0,0){$\ddots$}}
\put(70,90){\makebox(0,0){$\Lambda^1\otimes{\mathbb{V}}_0$}}
\put(120,60){\makebox(0,0){$\Lambda^1\otimes{\mathbb{V}}_1$}}
\put(170,30){\makebox(0,0){$\Lambda^1\otimes{\mathbb{V}}_2$}}
\put(190,15){\makebox(0,0){$\ddots$}}
\put(120,90){\makebox(0,0){$\Lambda_\perp^2\otimes{\mathbb{V}}_0$}}
\put(170,60){\makebox(0,0){$\Lambda_\perp^2\otimes{\mathbb{V}}_1$}}
\put(220,30){\makebox(0,0){$\Lambda_\perp^2\otimes{\mathbb{V}}_2$}}
\put(240,15){\makebox(0,0){$\ddots$}}
\put(220,60){\makebox(0,0){$\Lambda_\perp^2\otimes{\mathbb{V}}_0$}}
\put(270,30){\makebox(0,0){$\Lambda_\perp^2\otimes{\mathbb{V}}_1$}}
\put(290,15){\makebox(0,0){$\ddots$}}
\put(70,75){\makebox(0,0){$\uparrow\scriptstyle\partial$}}
\put(124,75){\makebox(0,0){$\uparrow\scriptstyle\partial_\perp$}}
\put(120,45){\makebox(0,0){$\uparrow\scriptstyle\partial$}}
\put(174,45){\makebox(0,0){$\uparrow\scriptstyle\partial_\perp$}}
\put(224,45){\makebox(0,0){$\uparrow\scriptstyle\partial_\perp^{(2)}$}}
\put(274,15){\makebox(0,0){$\uparrow\scriptstyle\partial_\perp^{(2)}$}}
\end{picture}\]
replacing~(\ref{E0level}), where $\partial_\perp^{(2)}$ is as in 
Proposition~\ref{replacementproposition}, which is now used together with 
Kostant's theorem~\cite{k} to identify the $E_1$-level as~(\ref{E1level}), 
just as before. The rest of the proof is unchanged. 
\end{proof}

\addtolength{\textheight}{18pt}
\section{Proof of the main theorem} The proof of Theorem~\ref{maintheorem} is
now a straightforward application of the machinery we have developed. Suppose
$\omega_{ab\cdots c}$ is a smooth symmetric $\ell$-tensor, globally defined on
${\mathbb{CP}}_n$ and having zero energy. Then the same is true of
$\iota^*\omega_{ab\cdots c}$ for any model embedding
$\iota:{\mathbb{RP}}_n\hookrightarrow{\mathbb{CP}}_n$. The X-ray transform 
on ${\mathbb{RP}}_n$ is well-understood and it is proved in~\cite{be} that
$\iota^*\omega_{ab\cdots c}$ is of the form $\nabla_{(a}\phi_{b\cdots c)}$. By 
Theorem~\ref{BGGonRP} we conclude that 
$\nabla^{(\ell)}(\iota^*\omega_{ab\cdots c})=0$. As this is true for all 
model embeddings, we conclude by Corollary~\ref{linearalgebra} that 
$\nabla_\perp^{(\ell)}(\omega_{ab\cdots c})=0$. 
Theorem~\ref{sufficiencyonCPn} 
finishes our proof.

\end{document}